\documentclass[11 pt]{amsart}
\usepackage[all]{xy}

\usepackage{amssymb,amsmath,color}

\newtheorem{theorem}{Theorem}[subsection]

\newtheorem{co}[theorem]{Corollary}
\newtheorem{lemma}[theorem]{Lemma}

\newtheorem{cor}[theorem]{Corollary}
\newtheorem{definition}[theorem]{Definition}
\newtheorem{prop}[theorem]{Proposition}
\theoremstyle{remark}
\newtheorem{remark}[theorem]{Remark}

\numberwithin{equation}{section}

\def\t{\textnormal}
\def\l{{\mathfrak l}}
\def\g{{\gamma}}

\def\O{{\omega}}

\newcommand{\BA}{{\mathbb A}}
\newcommand{\BQ}{{\mathbb Q}}
\newcommand{\BZ}{{\mathbb Z}}

\newcommand{\G}{{\Gamma}}
\newcommand{\SL}{{SL_2(\BZ)}}

 \let\a\alpha  \let\b\beta    \let\d\delta
  \let\g\gamma
  \let\l\lambda   
      
\let\GL\Lambda
\def\C{\mathbb C}
\def\G{\mathbf G}

\def\g{\gamma}
\def \gal{\textnormal{Gal}}
\def\GL{\mathbf{GL}}
\def \GL2 {{\text{GL}_2}}

\def \inj{\hookrightarrow}
\def \gl {\mathfrak l}
\def \gp{\mathfrak p}

\def\Tr{{\rm Tr}}
\def\Gal{{\rm Gal}}

\def\F{{\mathbb F}}
  
\def\Z{{\mathbb Z}}

\def\Q{{\mathbb Q}}

\def\O{\mathcal O}

\def\P{\wp}

\def\G{\Gamma}

\def\ord{{\mathrm ord}}

\begin{document}

\title[Galois representations with QM]{Galois representations with quaternion multiplication associated to noncongruence modular forms }
\author[A.O.L. Atkin]{A.O.L. Atkin $^1$}
\thanks{$^1$posthumous}
\author{Wen-Ching Winnie Li}
\address{Department of Mathematics, Pennsylvania State University,
University Park, PA 16802, USA ~and National Center for Theoretical Sciences, Mathematics Division,
National Tsing Hua University, Hsinchu 30013, Taiwan, R.O.C.} \email{wli@math.psu.edu}

\author{Tong Liu}
\address{Department of Mathematics, Purdue University, West Lafayette, IN 47907, USA}\email{tongliu@math.purdue.edu}

\author{Ling Long}
\address{Department of Mathematics, Iowa State University, Ames, IA
  50011, USA}
\email{linglong@iastate.edu}


\thanks{}
\thanks{
The second author is supported in part by the NSF grant DMS-0801096, the third author by the NSF grant DMS-0901360 and the fourth author by the NSA grant
H98230-08-1-0076 and the NSF grant DMS-1001332. Part of the paper was written when the fourth author was visiting the National Center for Theoretical Sciences in Hsinchu, Taiwan, and the University of California at Santa Cruz. She would like to thank both institutions for their hospitality.}

\maketitle
\begin{abstract}
In this paper we study the compatible family of degree-$4$  Scholl representations $\rho_{\ell}$
associated with a space $S$ of weight $\kappa>  2$ noncongruence cusp forms satisfying Quaternion Multiplication over
a biquadratic  extension of $\Q$. It is shown that  $\rho_\ell$ is automorphic, that is, its associated L-function has the same Euler factors as the L-function of an automorphic form for $\t{GL}_4$ over $\Q$. Further, it yields a relation between the Fourier coefficients of noncongruence cusp forms in $S$ and those of certain automorphic forms via the three-term Atkin and Swinnerton-Dyer congruences.
\end{abstract}

\tableofcontents
\section{Introduction}

To
a $d$-dimensional space $S_{\kappa}(\G)$ of cusp forms of weight $\kappa > 2$
for a noncongruence subgroup $\G$ under general
assumptions, in \cite{sch85b} Scholl attached a compatible family of
$2d$-dimensional $\ell$-adic representations $\rho_\ell$ of $G:=G_\Q=\gal(\overline \Q/\Q)$. Due to the motivic
nature of Scholl's construction, one
expects these representations to be  \emph{automorphic} in the sense that they are related to automorphic forms as predicted
by Langlands. For
general Scholl representations, this is too much to hope for at present, owing to the
 ineffective Hecke operators on noncongruence modular forms and the currently
available modularity techniques;  but it is achievable if the space
admits extra symmetries, as shown in \cite{lly05,all05,long061}. In each of these examples, the automorphy was established using the Faltings-Serre modularity technique, which becomes inefficient for general situation. On the other hand, tremendous progress in modularity has been made in recent years. In this paper, we prove the automorphy of degree-$4$ Scholl representations which admit quaternion multiplication by using modern technology. Moreover, the Atkin and Swinnerton-Dyer congruences for forms in the underlying space are also established.
\smallskip

Now we outline our method and main results. The recent settlement of Serre's
 conjecture over $\Q$ (cf.  Theorem \ref{serre}) by Khare, Wintenberger and
 Kisin (\cite{KW09}, \cite{kisin9}) is useful for establishing the automorphy of 2-dimensional
 representations of $G$ coming from geometry. In our situation, the
 Scholl representations $\rho_{\ell}$  were constructed from geometry, as a variation of  Deligne's construction of $\ell$-adic Galois representations for congruence cusp forms. They have  Hodge-Tate weights 0 and $1-\kappa$, each of multiplicity $d$, and all eigenvalues of the characteristic polynomial of a geometric Frobenius element are algebraic integers with the same complex absolute value. When $d=1$, by appealing to the now established Serre's conjecture, one concludes that $\rho_{\ell}$ arises from a congruence newform. When $d>1$, all existing potential automorphy criteria, such as \cite{BLTDR10}, assume the regularity condition that the Hodge-Tate weights are distinct, hence they cannot be applied to Scholl representations directly. In this paper we consider the case $d=2$ and the representation $\rho_\ell$ has Quaternion Multiplication (QM) over a field $K$. We prove that when $K$ is quadratic over $\Q$, after extending the scalar field,  $\rho_\ell$ decomposes into { the sum of two modular representations (Theorem \ref{QMoverquadratic}).}  In the more interesting case that $K$ is biquadratic over $\Q$, we show that $\rho_\ell$ is a tensor product of two 2-dimensional projective representations $\tilde \rho$ and $\tilde \g$ of $G$, in which $\tilde \g$ has finite image.
By using Tate's vanishing theorem ${\rm{H}}^2(G, \mathbb C^\times)= 0$, we lift $\tilde \g$ to an  (ordinary) representation $\g$ of $G$ with finite image. Consequently, $\tilde \rho$ can be lifted to a representation $\eta$ of $G$ such that $\rho_\ell=\eta\otimes\g$ is a tensor product of two degree $2$ representations of $G$. Of these $\g$ is induced from a finite character of the absolute Galois group $G_F$ of a quadratic field $F$ contained in $K$,
hence is automorphic. Moreover, it is odd if  $K$ is not totally real, and even otherwise { (Theorem \ref{thm:exist-chi_u} and Proposition \ref{oddness})}. The  other component $\eta$ is shown to be modular, arising from a weight $\kappa$ newform of a congruence subgroup { (Theorem \ref{thm:3.2.1})}. In fact, for each quadratic $F$ contained in $K$,  $\rho_\ell$ is induced from a degree $2$ representation of $G_F$ which comes from an automorphic form of $\rm{GL}_2(\mathbb A_F)$ { (Theorem \ref{QMdescent})}. Consequently, for $K$ biquadratic over $\Q$, the automorphy of $\rho_\ell$ can be seen in many ways: in addition to what is described above, it also corresponds to an automorphic representation of $\rm{GL}_2(\mathbb A_\Q) \times \rm{GL}_2(\mathbb A_\Q)$, which also implies that it comes from an automorphic representation of $\rm{GL}_4(\mathbb A_\Q)$ by a result of Ramakrishnan \cite{Ramakrishnan00} { (Remark 4.2.5)}.
\smallskip

 The automorphy of Scholl representations is
useful for understanding arithmetic properties of noncongruence modular forms. An application of the above modularity/automorphy result is an intriguing link between the coefficients of noncongruence cusp forms and congruence automorphic forms via Atkin and Swinnerton-Dyer congruences. See \S4.3 for details. Moreover,  automorphy and Atkin and Swinnerton-Dyer congruences combined can  be used to establish the $d=1$ case of the unbounded denominator conjecture, a fundamental characterization for the Fourier coefficients of genuine noncongruence modular forms (cf. \cite{ll10}). Our approach can be extended to handle other kinds of symmetries in a more general context, which will be dealt with in our future work.

   \smallskip

 This paper is organized in the following manner. In \S2, we obtain a simpler modularity criterion for 2-dimensional representation of $G$.  \S3 is devoted to the study of 4-dimensional Galois representations with QM over a biquadratic field.
 In \S4, we prove the automorphy of Scholl representations of $G$ for the cases $d=1$ and $d=2$ with QM, as well as the implications to Atkin and Swinnerton-Dyer congruences.  The main results are recorded in Theorems \ref{cor:d=1}, \ref{thm:3.2.1},  \ref{QMdescent} and \ref{ASD}. To illustrate our main results and methods, explicit examples are exhibited in \S5. We recast
 the known results on automorphy and ASD congruences obtained in \cite{lly05, all05,
 long061} in the framework of QM.  A new example of Scholl representations admitting QM over a biquadratic field is also given. The novelty of this example is the role played by the Atkin-Lehner involution.  The automorphy of Galois representations and the conjectural ASD congruences are established.
\smallskip

The authors  would like to thank  Henri Darmon and Siu-Hung Ng for enlightening  conversations,  Luis Dieulefait for communicating to us modularity results.

\subsection{Notation}
Given a number field  $F$, denote by $\O_F$ its ring of integers and
$G_F $ the absolute Galois group $ \gal (\overline F/F)$.  For any finite prime $\mathfrak l$ of $\O_F$
dividing the rational prime $\ell$,  write $F_\gl$ for the
completion of $F$ at $\gl$ and $k_\gl:= \O_F/ \gl$ its residue field. Its Galois group  $G_{F_\gl} := \gal(\overline \Q_\ell/F_\gl)$ is identified with a decomposition
group of $\gl$ in $G_F$.  Write
$I_{F_\gl}\subset G_{F_\gl}$ for the inertia subgroup over $\gl$ and
$\t{Fr}_\gl$ the \emph{arithmetic} Frobenius over $\gl$, which is the usual
topological generator of $G_{F_\gl}/I_{F_\gl} \simeq \gal(\bar k_\gl
/ k_\gl)$ sending $x \in \bar k_\gl$ to $x^{|k_\gl|}$. The \emph{geometric} Frobenius, which is
$\t{Fr}_\gl^{-1}$, is denoted by $\t{Frob}_\gl$.

We reserve $G$ for $G_\Q$, $G_\ell$ for $G_{\Q_\ell}$
and $I _\ell $ for $I_{\Q_\ell}$. For each $\ell$ we fix an embedding $\iota_\ell:
\overline \Q_\ell \inj \mathbb C$.  By an $\ell$-adic Galois representation over $E$
 we mean a continuous representation $\rho: H \rightarrow \text{Aut}_E(V)$ where $H$ is a subgroup
of $G$ and $V$   a finite-dimensional vector space over $E$,  a finite extension of $\Q_\ell$. The reference to the $\ell$-adic field $E$ will be dropped when it is not important. We  refer $\rho$ to  $V$
if no confusions arise.  {Denote by $\rho^\t{ss}$ or $V^\t{ss}$ the semi-simplification of $\rho$.
Let $ \rho$ be a Galois
representation of $H$ as above. If $H$ is compact then there exists an
$\O_{E}$-lattice $T$ in $V$ {  invariant under $\rho(H)$.} Let $\mathfrak m_E$ be the
maximal ideal of $\O_E$. We get a \emph{residual representation}
$\bar \rho$ of $H$ on the vector space $T/ \mathfrak m_E T$ over the residue
field $k: = \O_E/ \mathfrak m_E$. Although $\bar \rho$ depends on the choice of
$T$, its semi-simplification $\bar \rho^{\t{ss}} $ does not. }  Two continuous linear
representations $\rho_i$,  $i=1, 2$, of the topological
group $H$ acting on finite-dimensional $E_i$-vector spaces $V_i$, where $E_1$ and $E_2$ are contained in a
finite common extension $ E$,
are said to be  \emph{equivalent}, denoted $\rho_1 \sim \rho_2$,
if $\rho _1 \otimes _{E_1}  E \simeq \rho _2 \otimes _{E_2}
  E $ as representations of $H$.

\section{Modularity of { degree two} $\ell$-adic Galois representations}
In this section,  we derive from known
modularity results for 2-dimensional $\ell$-adic Galois representations of $G$
 a useful modularity criterion for applications in later sections.

\subsection{Modularity of 2-dimensional $\ell$-adic Galois representations}\label{moduarlity}
 Let $\kappa \geq 2$ and  $N\geq 1$ be integers and $S_\kappa(\G_1(N), \C)$ be the space of cusp form{s} of weight $\kappa$ and level $N$. Suppose that $f= \sum \limits_{i=1 }^\infty a_n q^n$ is a newform normalized with $a_1=1$. The following is a classical result {proved by Deligne (\cite{De}) when $\kappa>1$, and Deligne and Serre when $\kappa=1$ (\cite{De2}).
\begin{theorem}\label{Deligne}
For a newform $f$ as above, the field of coefficients $E_f= \Q(a_n, {n\geq 1}) \subset \C$ is a number field.
Moreover,   for any prime $\lambda \mid \ell$ of $E_f$, there exists a
continuous representation
$$\rho_{f, \lambda} : G~ \rightarrow  ~\GL2 (E_{f, \lambda} ),$$
{where $E_{f, \lambda} $ is the completion of $E_f$ at $\lambda$,} such that
\begin{enumerate}
\item $\rho_{f, \lambda}$ is odd and absolutely irreducible;

\item For any $p \nmid N\ell$, $\rho_{f, \l}$ is unramified at $p$ and $\Tr(\rho_{f, \lambda}(\textnormal{Fr}_p)) =a_p$;
\item If ${ \ell} \nmid N$, then $\rho_{f, \lambda}|_{G_\ell}$ is crystalline with Hodge-Tate weights $\{0, k-1\}$.

\end{enumerate}
\end{theorem}
In what follows, an $\ell$-adic representation $\rho$ of $G$ is said to be
 \emph{modular} if there exists a
modular form $f \in S_\kappa(\Gamma_1(N), \C)$ and a prime $\lambda
\mid \ell$ of $E_f$ such that $\rho_{f, \lambda} \sim \rho$ {  or its dual $\rho^\vee$}.

Recall that for any prime $p$, $G_p \subset G$ denotes a
decomposition group at $p$ and $I _p \subset G_p$ is the inertia
subgroup. Let $\bar \rho: G \to \t{GL}_2(k)$ be a representation over a finite { field $k$ of characteristic $\ell$.} Call
$\bar \rho$ \emph{modular} if there exists a modular form $f \in
S_\kappa(\Gamma_1 (N), \C)$  and a prime $\lambda\mid \ell$  of $E_f$ such
that the residual representation $\bar \rho_{f, \lambda}$ is
equivalent to $\bar \rho$.

The following has been conjectured by Serre and proved by Khare,
Wintenberger and Kisin (\cite{KW09,kisin9}}). We refer to \cite{kisin5} for a nice review
of Serre's conjecture.
\begin{theorem}[Serre's conjecture] \label{serre} Any odd and absolutely irreducible representation $\bar \rho: G \to \t{GL}_2 ( k)$ is modular.
\end{theorem}
\begin{remark} The precise Serre's conjecture also predicts the (minimal) weight and the level of the modular form $f$, which are not needed here.
\end{remark}

The aim of this subsection is to reprove the following result.
\begin{theorem}[\cite{DM03}, \cite{Die08}]\label{tool}
Let $E/ \Q_\ell$ be a finite extension. Suppose that $\rho : G \to
\t{GL}_2 (\O_E)$ is an $\ell$-adic Galois representation such that
\begin{enumerate}
\item $\rho$ is odd and absolutely irreducible;
\item $\rho$ is unramified at almost all primes;
\item $\rho|_{G_\ell}$ is crystalline with Hodge-Tate weights $\{0, r\}$   such that  $1 \leq r \leq \ell-2$ and $\ell+1 \nmid 2r.$
\end{enumerate}
Then $\rho$ is modular.
\end{theorem}

This theorem was essentially proved in \cite{DM03}, as explained in \cite{Die08}. We sketch the proof below
because some ingredients will be used later. We remark that the above modularity lifting theorem is an easy version (compared to that of \cite{Kis}), however it is enough for the applications in this paper.

The proof will use the discussion of local representation at $\ell$ and
several modularity lifting theorems. First we
discuss the local representation $ \rho|_{G_\ell}$ and its
reduction.  For $i \ge 1$, the map
$$\omega_i: I _\ell \to \F_{\ell^i}^\times  {\rm ~~defined ~by}   ~~g \mapsto \frac{g
(\sqrt[\ell^i-1]{\ell})}{\sqrt[\ell^i-1]{\ell}} \mod \ell $$ is the fundamental
character of level $i$. Note that $\omega_1= \epsilon_\ell \mod \ell$,
where $\epsilon_\ell$ is the $\ell$-adic cyclotomic character.

Since $\rho|_{G_\ell}$ is crystalline with Hodge-Tate weights $\{0, r\}$ and
$r \leq \ell-2$, there are two possibilities:

Type I:  $\rho|_{G_\ell}$ is absolutely reducible. In this case, $\rho
|_{I_\ell} \sim  \begin{pmatrix} \epsilon_\ell^{r}  & * \\ 0 & 1
\end{pmatrix}$ and
$(\bar \rho \otimes_k {\bar k})|_{I_\ell}  \simeq \begin{pmatrix}  \omega_1 ^{r} & * \\ 0 & 1 \end{pmatrix}.$

Type II: $\rho|_{G_\ell}$ is absolutely irreducible. In this case, $\bar
\rho \otimes_k \bar k |_{I_\ell}\simeq \begin{pmatrix}\omega_2 ^{r} & 0
\\ 0 & \omega_2^{\ell r}
\end{pmatrix}.$ In particular, $(\bar \rho \otimes_k \bar k)|_{G_\ell}$ is irreducible because $\t{Fr}_\ell$ will ``swap" $\omega_2$ and $\omega_2^\ell$.
\emph(see the proof of case (3) for the precise statement).

 The reader is referred to \S 4.2.1 in \cite{B&M} for the proof of the above statements. We also need the following result on Galois characters.

\begin{lemma}\label{character}Let $\chi: G \to \O_E^{\times}$ be an $\ell$-adic Galois character which is unramified outside a  finite set $S$ of finite primes excluding  $\ell$. Then $\chi$ is a finite character
in the sense that the image of $\chi$ is a finite (cyclic) group.
\end{lemma}
\begin{proof}  Let $p$ be a prime in $S$. It suffices to show that  if $\chi$ is ramified at $p$ then the image $\chi (I_p)$ is finite. Let $I^w_p$ be the wild inertia subgroup of $I_p$, $I^t_p:= I_p/I^w_p$ the tame inertia group, and  $G^w := \chi (I ^w_p)$.
We claim that the map $\chi:G^w \hookrightarrow \O^{\times}_E \overset q
\to (\O_E/ \mathfrak m_E)^{\times}= k^{\times}$ is an injection. In fact,
 $G^w$ is a pro-$p$-group and $\t{ker}(q)$ is a pro-$\ell$-group.
So they can only have trivial intersection. Since $G^w$ injects in
$k^{\times}$, $G^w$ is a finite group. Replacing $\Q$ by a suitable finite
extension, we may assume that  $\chi$ factors though $I^t_p$. Let
$\tau$ be a lift of $\t{Fr}_p$  and $\sigma \in I^t_p $. We have
$\tau \sigma  = \sigma^p \tau$. Applying $\chi$ to this equation, we
see that $\chi(\sigma)^{p-1}= 1$. Thus ${\chi}(I^t_p)$ is contained in the
group of $(p-1)$th-units in $\O_E^{\times}$, hence $\chi(I_p)$ is finite. \end{proof}

\begin{co}\label{character2} Assume that $\chi: G \to \O_E^{\times}$ is a Galois character  such that $\chi$ is unramified almost everywhere
and $\chi|_{G_\ell}$ is crystalline. Then there exists a finite
character $\psi$ and an  integer $r$ such that $\chi = \psi
\epsilon_\ell ^r$.
\end{co}
\begin{proof} Using the $p$-adic Hodge theory on the classification of crystalline
characters of $G_\ell$, we see that $\chi|_{G_\ell} = \psi_\ell \epsilon_\ell^r $
with $\psi_\ell$  a character  unramified at $\ell$. Applying the above lemma to
$\chi \epsilon_\ell^{-r}$,  we prove the corollary.
\end{proof}
\begin{remark}\label{rem:2.2.6} The above corollary may fail if $G$ is replaced by $G_K$
with $K/\Q$ a finite extension.  The problem is  that the
classification of crystalline characters of $G_{\gl}$ is much more
complicated if $\ell$ is inert in $K$. Here is a more
concrete example: Let $K$ be an imaginary quadratic extension of
$\Q$  and consider an elliptic curve $E$ defined over $K$ with
complex multiplication by the ring $\O_K$. Choose a prime ideal $\gl$ of
$\O_K$ generated by a prime $\ell$ inert in $K$.
  Then the
Tate module $T_\ell (E)$ induces a Galois $\O_{K_\gl}$-character $\chi:
G_K \to \O^{\times}_{K_\gl}$. Note that $\chi$ has Hodge-Tate weights 0 and
1. So it can not be written as $\psi \epsilon_\ell^r$, which only has
Hodge-Tate weight $r$.
\end{remark}

Let $k = \O_E/ \mathfrak m_E$ denote the residue field of
$\O_E$, { where $\mathfrak m_E$ is the maximal ideal of $\O_E$.} Let $\bar \rho: G \to \t{GL}_2(k)$ be the reduction of
$\rho$. We distinguish three cases:
 \begin{enumerate}
 \item $(\bar \rho\otimes_k \bar k)|_{G_\ell} $ is reducible but
 $\bar \rho\otimes_k \bar k $ is irreducible;
 \item $\bar \rho\otimes_k \bar k $ is reducible;
 \item $(\bar \rho\otimes_k \bar k)|_{G_\ell} $ is irreducible.
  \end{enumerate}

Now we proceed to prove Theorem \ref{tool} case by case. For case (1), $(\bar \rho\otimes_k \bar
k)|_{G_\ell} $ is reducible, hence it
is of type I. So $\rho |_{I_\ell} \sim  \begin{pmatrix}
\epsilon_\ell^{r} &
* \\ 0 & 1  \end{pmatrix}$.  Then the main theorem in the
introduction of \cite{SW2} and Theorem \ref{serre} imply  that $\rho$ is modular.
\smallskip

Next we consider case (2). Note that  $\bar \rho \otimes_k \bar k$ is
reducible, so is $(\bar \rho \otimes_k \bar k)|_{G_\ell}$. As discussed above,  $\rho _\ell$ is of type I and
the semi-simplification $(\bar \rho \otimes_k \bar k)^\t{ss}|_{I_\ell}
= 1 \oplus \omega_1^r$. Hence $(\bar \rho \otimes_k \bar k)^\t{ss} =
\bar \psi \oplus \chi$ with $\bar \psi|_{I_\ell}= 1$.    Twisting
$\rho$ by a finite character if necessary, we may assume that $(\bar
\rho \otimes_k \bar k)^\t{ss} = 1 \oplus \chi$ with $\chi|_{I_\ell}=
\omega_1 ^{r}\not =1$. So the Theorem in the
introduction of \cite{SW2} proves that $\rho$ is modular.

\smallskip

Finally,   case (3) follows  from Theorem 0.3 in the introduction of \cite{dml} and Theorem \ref{serre}. Note that the condition $1 \leq r \leq \ell-2$ and $\ell+1 \nmid r$  implies that $\bar \rho$ restricted to $G_{\Q\big(\sqrt{(-1)^{(\ell-1)/{2}}\ell }\big)}$  is absolutely irreducible, which is required by  Theorem 0.3 in the introduction of \cite{dml}.

\section{Galois representations endowed with quaternion multiplication}

In this section we show that if a $4$-dimensional $\ell$-adic representation of $G$ is endowed with quaternion multiplication over a quadratic or biquadratic field, then either it decomposes into the sum of two degree $2$ representations, or
it is induced from a degree $2$ representation of an index $2$ subgroup $G_K$ of $G$. In the latter case, we show that this degree $2$ representation of $G_K$, after twisting by a character, can be extended to a representation of $G$.

\subsection{Quaternion multiplication}
If there is a quaternionic action on the space of a $4$-dimensional $\ell$-adic representation in the following sense, then a lot more can be said. In this section, $F$ is always assumed to be a finite extension of $\Q_\ell$.
 \medskip
\begin{definition} \label{QM}  Let   $\rho_\ell$ be an $\ell$-adic representation of $G$ acting on a $4$-dimensional $F$-vector space $W_\ell$. It is said to have quaternion multiplication (QM) if
there are linear operators $J_s$ and $J_t$ on $W_\ell$,
  parametrized by two distinct non-square integers $s$ and $t$,  satisfying
\begin{itemize}
\item[(a)] $J_s^2 = J_t^2 = -id,  J_sJ_t = -J_tJ_s$;

\item[(b)] For $u \in \{s, t\}$ and  $g \in G$, we have { $J_u \rho_\ell(g) = \pm \rho_\ell(g) J_u$ with $+$ sign if and only if $g  \in G_{\Q(\sqrt u)}$}.
\end{itemize}

\end{definition}
In this case, we say that the representation has QM over $\Q(\sqrt s, \sqrt t)$.

\begin{theorem}\label{quaternion}
Let $\rho_\ell$ be a $4$-dimensional $\ell$-adic representation over $F$ with QM over $\Q(\sqrt s, \sqrt t)$.
Then the following two statements hold.
\begin{itemize}
\item[(1)]  For a nonsquare $u \in \{s, t, st\}$, there  are two $2$-dimensional $\ell$-adic representations $\sigma_u$  and $\sigma_u^-$ of $G_{\Q(\sqrt u)}$  over  $E = F(\sqrt {-1})$ such that $$\rho_\ell \otimes_{F}E \simeq \textnormal{Ind}_{G_{\Q(\sqrt u)}}^{G} \sigma_u,  \quad
  (\rho_\ell\otimes_{F}E)|_{G_{\Q(\sqrt{u})}}= \sigma_{u} \oplus  \sigma_{u}^{-}, \quad {\rm and} \quad {\sigma_u^- \cong } ~ \sigma_{u}\otimes
\delta_{u} \cong \sigma_{u}^{\tau},$$ where $\delta_{u}$ is the character of $G_{\Q(\sqrt{u})}$ with kernel
$G_{\Q(\sqrt{s},\sqrt{t})}$ and $\tau \in G \setminus G_{\Q(\sqrt u)}$. Consequently, If $\Q(\sqrt{s},\sqrt{t})$ is a biquadratic extension of $\Q$, $\rho_\ell\otimes _{F}E $ is induced from three representations $\sigma_u$ of $G_{\Q(\sqrt u)}$ for $u \in \{s, t, st\}$. Further $\sigma_u$ is irreducible if $\rho_\ell\otimes_{F}E$ is.

\item[(2)] If $\Q(\sqrt{s},\sqrt{t})$ is a quadratic extension of $\Q$, then there is a $2$-dimensional $\ell$-adic representation  $\sigma$ of $G$ over  $E= F(\sqrt {-1})$ such that  $$\rho_\ell\otimes_{F}E=\sigma \oplus  (\sigma \otimes
\theta),$$
where $\theta$ is the
quadratic character of $G$ with kernel $G_{\Q(\sqrt{s},\sqrt{t})}$. In particular, $\rho_\ell$ is  absolutely reducible.
\end{itemize}

\end{theorem}

There is a vast literature on abelian varieties attached
to weight 2 congruence Hecke eigenform or more general modular
motives possessing quaternion multiplication (QM) due to the
existence of extra twists given by operators like Atkin-Lehner
involutions (see \cite{ribet80-twist},  \cite{momose81},
\cite{brown-ghate03}, \cite{ghate-G-J-Quer}, \cite{Dieulefait03} etc). Compared to these
cases our situation distinguishes itself by the specific condition
on the field of definition for each $J$ operator.

\begin{proof} We follow the  proof of Theorem 6 in \cite{all05}. Set $J_{st}= J_sJ_t$. It is easy to check that $J_{st}$ satisfies property (b) in Definition \ref{QM} and $J_{st}^2 = -id$. Let $E = F(\sqrt {-1})$ and write $\rho_\gl$ for $\rho_\ell\otimes_{F}E$ for simplicity.
Given a nonsquare $u \in \{s, t, st\}$, we first show that there exists an $\ell$-adic $2$-dimensional representation $\sigma_u$ of $G_{\Q(\sqrt u)}$ such that $\rho_\gl \simeq \t{Ind}^{G}_{G_{\Q(\sqrt u)}}\sigma_u$. Since $J_u^2 = -id$, its eigenvalues are contained in $E$. Let  $v \in W_\ell \otimes_{F}E$ be an eigenvector of $J_u$ with eigenvalue $i$, a square root of $-1$. Then property (b) implies that for $g \in G \setminus G_{\Q(\sqrt u)}$, $\rho_\gl(g) v$ is an eigenvector of $J_u$ with the opposite eigenvalue $-i$. Thus $J_u$ has eigenvalues $\pm i$ with  $2$-dimensional $\pm i$-eigenspace. Since $J_u$ commutes with the action of $G_{\Q(\sqrt u)}$, each $\pm i$-eigenspace affords a { $G_{\Q(\sqrt u)}$-action}, denoted by $\sigma_u$, $\sigma^-_u$, respectively. When $u$ is a nonsquare, by property (b), any $\tau \in G \setminus G_{\Q(\sqrt u)}$ gives rise to an $E$-isomorphism from the $(-i)$-eigenspace to the $i$-eigenspace. On the $(-i)$-eigenspace, for all $g \in G_{\Q(\sqrt u)}$, we have $\sigma_u^{-}(g)= \rho_{\gl}(g) = \rho_{\gl}(\tau^{-1} \tau g \tau^{-1} \tau) = \rho_{\gl}(\tau)^{-1} \sigma_u^{\tau}(g) \rho_{\gl}(\tau)$, which shows that $\sigma_u^{-}$ and $\sigma_u^\tau$ are isomorphic $G_{\Q(\sqrt u)}$-modules. Therefore $\rho_\gl \simeq \t{Ind}^{G}_{G_{\Q(\sqrt u)}}\sigma_u$.
Since $\rho_\gl$ is induced from $\sigma_u$, if $\rho_\gl$ is irreducible, so is $\sigma_u.$

Let $\{v_1, v_2\}$ be a basis of the $i$-eigenspace of $J_u$ and choose a nonsquare $v \in \{s, t, st\}$ not equal to $u$.  Then $v_3 :=J_vv_1$ and $v_4 :=J_v v_2$ form a basis of the $-i$-eigenspace of $J_u$. With respect to the ordered basis $\{v_1, v_2, v_3, v_4\}$ we may express the operators by the matrices
\begin{equation}\label{eq:gamma}
  J_u=\begin{pmatrix}
    iI_2&0\\0&-iI_2
  \end{pmatrix}, ~J_v=\begin{pmatrix}
    0&-I_2\\I_2&0
  \end{pmatrix}, ~J_vJ_u=\begin{pmatrix}
    0&iI_2\\iI_2&0
  \end{pmatrix}.\end{equation} This gives a representation $\gamma$ of the quaternion group generated by $J_s$ and $J_t$.

  First we prove (1).
  Let $N=G_{\Q(\sqrt{s},\sqrt{t})}$. It follows from (b) that with respect to the same basis, $\rho_\gl(g)$ is represented by $
  \begin{pmatrix}
    P(g)&0\\0&P(g)
  \end{pmatrix}$ for $g \in N$,  $
  \begin{pmatrix}
    P(g)&0\\0&-P(g)
  \end{pmatrix}$   for $g \in G_{\Q(\sqrt{u})}\setminus N$,  $
  \begin{pmatrix}
    0&P(g)\\-P(g)&0
  \end{pmatrix}$   for $g \in G_{\Q(\sqrt{v})}\setminus N$, and $
  \begin{pmatrix}
    0&P(g)\\P(g)&0
  \end{pmatrix}$   for $g \in G_{\Q(\sqrt{vu})}\setminus N$.
This shows that the map sending $g \in G_{\Q(\sqrt u)}$
to the matrix $P(g)$ is  the $2$-dimensional representation $\sigma_u$ of   $G_{\Q(\sqrt u)}$ and  the representation $\sigma_u^-$ is given by $\sigma_u \otimes \delta_u$,   where $\delta_{u}$ is the character of $G_{\Q(\sqrt{u})}$ with kernel
$G_{\Q(\sqrt{s},\sqrt{t})}$.
Next we prove (2). Since $\Q(\sqrt s, \sqrt t) = \Q(\sqrt t) = \Q(\sqrt s)$, we have $N=G_{\Q(\sqrt{s},\sqrt{t})} =G_{\Q(\sqrt t)} $.
In this case $\rho_\gl(g)$ is represented by $
  \begin{pmatrix}
    P(g)&0\\0&P(g)
  \end{pmatrix}$ for $g \in N$ and  $
  \begin{pmatrix}
    0&P(g)\\P(g)&0
  \end{pmatrix}$   for $g \in G \setminus N$.
Furthermore, $\{v_1+v_3, v_2+v_4\}$  is a basis of the
$i$-eigenspace $W$ of $J_sJ_t$, and $\{v_1-v_3, v_2-v_4\}$  is a
basis of the $(-i)$-eigenspace $W^-$ of $J_sJ_t$. Each space is
$G$-invariant by property (b) since $\Q(\sqrt s) = \Q(\sqrt t)$;
denote the $2$-dimensional representations on $W$ and $W^-$ by
$\sigma$ and $\sigma^-$, respectively. It is straightforward to
check that with respect to the above bases of $W$ and $W^-$, we have
$\sigma(g) = P(g) = \sigma^-(g)$ for $g \in N$ and $\sigma(g) = P(g) =
-\sigma^-(g)$ for $g \in G \setminus N$. This shows that
$\sigma^- = \sigma \otimes \theta$, where $\theta$ is the quadratic
character of $G$ with kernel $G_{\Q(\sqrt s, \sqrt t)}$.
  \end{proof}

If $\Q(\sqrt{s},\sqrt{t})$ is a quadratic extension of $\Q$,  the characteristic polynomial
$H_p(x)$ of $\rho_\ell({ \t{Frob}}_p)$ is the product of the characteristic polynomials of
$\sigma({ \t{Frob}}_p)$ and $(\sigma\otimes \theta)({ \t{Frob}}_p)$.
We shall see that there is also a natural factorization for $H_p(x)$ when $\Q(\sqrt{s},\sqrt{t})$ is biquadratic.

{\begin{remark} \label{charpoly}  Assume $\Q(\sqrt{s},\sqrt{t})$ is biquadratic over $\Q$. As stated in the theorem above, for any $\tau \in G \setminus G_{\Q(\sqrt u)}$, we have
 $\sigma_u^{\tau}  \simeq \sigma_u \otimes \delta_u$. Thus for any prime $p   \ne \ell$ splitting into two places $\gp_{\pm}$ in $\Q(\sqrt u)$, $\tau$ permutes the two places $\gp_{\pm}$; and if $\rho_\ell$ is unramified at $p$, the characteristic polynomial of $\sigma_u^\tau({ \t{Frob}}_{\gp_+})= \sigma_u({ \t{Frob}}_{\gp_+})\delta_u({ \t{Frob}}_{\gp_+})$ is equal to that of $\sigma_u({ \t{Frob}}_{\gp_-})$. Therefore
  $\sigma_u({ \t{Frob}}_{\gp_{\pm}})$ have the same characteristic polynomials
 if $\delta_u({ \t{Frob}}_{\gp_{\pm}}) = 1$, which occurs if and only if $p$ splits in $\Q(\sqrt s, \sqrt t)$.
\end{remark}

The statement (a) below on the factorization of the characteristic polynomial of $\rho_\ell({ \t{Frob}}_p)$ at an unramified place $p$ follows immediately from Theorem \ref{quaternion} and the remark above.

\begin{cor}\label{factorization} Let $\rho_\ell$ be a $4$-dimensional $\ell$-adic representation of $G$ with QM over $\Q(\sqrt s, \sqrt t)$, a biquadratic extension of $\Q$. For $u \in \{s, t, st\}$, let $\sigma_u$ be as in Theorem \ref{quaternion}.
\begin{itemize}
\item[(a)] Let $p$ be a prime different from $\ell$ at which $\rho_\ell$ is unramified. Then the degree-$4$ characteristic polynomial $H_p(x)$ of $\rho_\ell ({ \t{Frob}}_p)$  and the degree-$2$ characteristic polynomial $H_{\gp, u}(x)$ of $\sigma_u({ \t{Frob}}_{\gp})$ at a place $\gp$ of $\Q(\sqrt u)$ dividing $p$ are related as follows:
\begin{itemize}
\item[(a1)] If $p$ splits completely in $\Q(\sqrt s, \sqrt t)$, then $H_p(x) = H_{\gp, u}(x)^2$, where $H_{\gp, u}(x) = x^2 - A(p)x + B(p)$  is independent of the choice of $\gp$ and $u$;
\item[(a2)] If $p$ splits in $\Q(\sqrt u)$ and not in $\Q(\sqrt s, \sqrt t)$, then there are two places $\gp_{\pm}$ of $\Q(\sqrt u)$ dividing $p$, and $H_p(x) = H_{\gp_+, u}(x) H_{\gp_-, u}(x) $, where $H_{\gp_{\pm}, u}(x) = x^2 \mp A_u(p)x + B_u(p)$.
\item[(a3)] If $p$ is inert in $\Q(\sqrt u)$, then there is one place $\gp$ of $\Q(\sqrt u)$ dividing $p$, and $H_p(x) = H_{\gp, u}(x^2)$.
\end{itemize}

\item[(b)] Suppose $u > 0$, $c \in G_{\Q(\sqrt u)}$ is a complex conjugation, and $\tau \in G \setminus G_{\Q(\sqrt u)}$. Then $\Tr \sigma_u(c) = \Tr \sigma_u^\tau(c)$ if $\Q(\sqrt s, \sqrt t)$ is totally real, and $\Tr \sigma_u(c) = - \Tr \sigma_u^\tau(c)$ otherwise.
\end{itemize}
\end{cor}

\begin{proof} It remains to prove part (b). This follows from the fact that $\sigma_u^\tau(c) = \sigma_u(c) \delta_u(c)$ and $\delta_u(c) = 1$ if and only if $\Q(\sqrt s, \sqrt t)$ is totally real.
\end{proof}

 Consequently, in case (a2) above,  $H_p(x)$ is also the characteristic polynomial of $\rho_\ell({ \t{Frob}}_{\gp_{\pm}})$. Since a prime $p$ splits in at least one of the quadratic {fields contained in the biquadratic field $\Q(\sqrt s,\sqrt t)$}, when the representation $\rho_\ell$ admits QM over $\Q(\sqrt s, \sqrt t)$, we obtain a natural factorization of the characteristic polynomial of $\rho_\ell ({ \t{Frob}}_p)$ as a product of two quadratic polynomials.
\smallskip

\begin{prop}\label{conjugation} Keep the same notation and assumptions as in Theorem \ref{quaternion}.
Assume that
$-1$ is not a square in $F$.
Let $j$ denote a square root of $-1$ and set $E= F(j)$.
\begin{itemize}
\item[(a)] For any $g \in G_{\Q(\sqrt u)}$, let $H_{g,+}(x)$ and $H_{g,-}(x)$ be the characteristic polynomials
of  $\sigma_u(g)$ and $\sigma_u^-(g)$ respectively. Then $H_{g,+}(x)$ and $H_{g,-}(x)$ are conjugate under the map $j \mapsto -j$.

\item[(b)] Assume $\Q(\sqrt s, \sqrt t)$ is biquadratic over $\Q$. For $u$, $c$ and $\tau$ as in Corollary \ref{factorization} (b), we have $\Tr \sigma_u(c) = \Tr \sigma_u^\tau(c)$. Therefore $\sigma_u$ is odd if $\Q(\sqrt s, \sqrt t)$ is not totally real.

\item[(c)] Assume $\Q(\sqrt s, \sqrt t)$ is quadratic over $\Q$. Then $\Tr \sigma(c) = \Tr \sigma(c) \theta(c)$ for the complex conjugation $c \in G$.
\end{itemize}
Consequently, if $\Tr \rho_\ell(c) = 0$ at the complex conjugation $c$, then $\sigma$ is odd when $\Q(\sqrt s, \sqrt t)$ is quadratic, and $\sigma_u$ is odd for positive $u \in \{s, t, st\}$  when $\Q(\sqrt s, \sqrt t)$ is biquadratic.
\end{prop}

\begin{proof} (a)
Let $w$ be a nonzero vector in  the representation space ${W_{\ell}}$ of $\rho_\ell$. Then $w$ and $J_u w$ are linearly independent over $F$ since the eigenvalues of $J_u$ are outside $F$ by assumption. Let $w'$ be a vector in ${W_{\ell}}$ and not in the $F$-span $\langle w, J_u w \rangle$. Claim that $w, w', J_u w, J_u w'$ are linearly independent over $F$. If not, then $J_u w' = \alpha w + \beta J_u w + \gamma w'$ for some $\alpha, \beta, \gamma \in F$ not all zero. Apply $J_u$ to the above relation to get another relation $\gamma J_u w' = \beta w - \alpha J_u w - w'$. Comparing both relations yields $\gamma^2 = -1$, a contradiction. Therefore $\{w, w', J_u w, J_u w'\}$ forms an $F$-basis of ${W_{\ell}}$. With respect to this ordered basis, $J_u$ is represented by the block matrix
$$ \begin{pmatrix} 0 & -I_2 \\ I_2 & 0 \end{pmatrix}.$$

On the extension ${W_{\ell}}\otimes_F E$ the actions of $\rho_\ell(G)$ and $J_u$ are $E$-linear. Recall that $\sigma_u$ and $\sigma^-_u$ are $\rho_\ell$ restricted to the $\pm j$-eigenspaces of $J_u$ on ${W_{\ell}} \otimes_F E$. It is easy to check that
$\{w- j J_u w,  w' - j J_u w'\}$ forms an $E$-basis for the space of $\sigma_u$ and $\{w+ j J_u w,  w' + j J_u w'\}$ forms a basis for the space of $\sigma^-_u$. In other words,
$$ \begin{pmatrix}I_2 & -jI_2 \\ I_2 & j I_2 \end{pmatrix} \begin{pmatrix} 0 & -I_2 \\ I_2 & 0 \end{pmatrix}
 \begin{pmatrix}I_2 & -jI_2 \\ I_2 & j I_2 \end{pmatrix}^{-1} = \begin{pmatrix} j I_2 & 0 \\ 0 & -j I_2 \end{pmatrix}.$$

Since $J_u$ commutes with $\rho_\ell(G_{\Q(\sqrt u)})$ by condition (b) of QM, given any $g \in G_{\Q(\sqrt u)}$,  the matrix of $\rho_\ell(g)$ with respect to $\{w, w', J_u w, J_u w'\}$ is represented by the block matrix
$$\begin{pmatrix} P & R \\ -R & P \end{pmatrix}$$
for some $2 \times 2$ matrices $P$ and $R$ with entries in $F$. One checks that
$$ \begin{pmatrix}I_2 & -jI_2 \\ I_2 & j I_2 \end{pmatrix} \begin{pmatrix} P & R \\ -R & P \end{pmatrix}
 \begin{pmatrix}I_2 & -jI_2 \\ I_2 & j I_2 \end{pmatrix}^{-1} = \begin{pmatrix} P+jR & 0 \\ 0 & P-jR \end{pmatrix}.$$
 This shows that the action of $\rho_\ell(g)$ on the $\pm j$-eigenspaces of $J_u$ are represented by matrices conjugate under $j \mapsto -j$, hence the characteristic polynomials ${H_{g,\pm}(x)}$ are as asserted.

(b) For $u > 0$, a complex conjugation $c$ lies in $G_{\Q(\sqrt u)}$. By applying (a) to $g = c$, we get that $\Tr \sigma_u(c)$ and $\Tr \sigma_u^\tau(c)$ are conjugate under $j \mapsto -j$. As $c^2 = id$, the possible values for $\Tr \sigma_u(c)$ are $\pm 2$ and $0$, which all lie in $F$. Hence $\Tr \sigma_u(c) = \Tr \sigma_u^\tau(c)$. When $\Q(\sqrt s, \sqrt t)$ is not totally real, it follows from Corollary \ref{factorization} (b) that the two traces are opposite, therefore they are equal to zero, and hence $\sigma_u$ is odd.

(c) In this case, choose $u = st$, which is a square. The operator $J_u = J_{st}:= J_s J_t$ commutes with $\rho_\ell(G)$. The representation $\sigma_u$ is called $\sigma$ and $\sigma_u^-$ called $\sigma \otimes \theta$ in Theorem \ref{quaternion}.
Apply (a) to $g = c$, the complex conjugation in $G$. By the same argument as in (b) we conclude
 $\Tr \sigma(c) = \Tr \sigma(c)\theta(c)$.

Finally, using $0 = \Tr \rho_\ell(c) = \Tr \sigma(c) + \Tr \sigma(c)\theta(c) = 2 \Tr \sigma(c)$ when  $\Q(\sqrt s, \sqrt t)$ is quadratic, and $0 = \Tr \rho_\ell(c) = \Tr \sigma_u(c) + \Tr \sigma_u^\tau(c) = 2 \Tr \sigma_u(c)$ for positive $u \in \{s, t, st\}$ when
$\Q(\sqrt s, \sqrt t)$ is biquadratic, we conclude the oddness of $\sigma$ and $\sigma_u$ respectively.
\end{proof}

{

\subsection{Quaternion multiplication and extension of representations} Let $K$ be a finite extension of $\Q$ and  $V$ be an $\ell$-adic representation of $G_K$.  We say that $V$ can be \emph{extended} to an $\ell$-adic  representation $V'$ of $G$ if $V'|_{G_K} \sim V$. Note that $V'$, if exists, is not unique in general.

Assume that $\rho_\ell$ has QM over a biquadratic $\Q(\sqrt s, \sqrt t)$. By Theorem \ref{quaternion}, $\rho_\ell \sim \textnormal{Ind}_{G_{\Q(\sqrt u)}}^G \sigma_u$ for $u \in \{ s, t, st\}$.
We shall show that, for each $u \in \{ s, t, st\}$, there exists a finite character $\chi_u$ of
$G_{\Q(\sqrt{u})}$ and a 2-dimensional representation $\eta_u$ of $G$ such that $\sigma_u\otimes \chi_u\sim\eta_u|_{G_{\Q(\sqrt{u})}}$. In other words, $\sigma _u \otimes \chi_u$ can be extended to $\eta_u$.

To see this,  we return to the proof of Theorem \ref{quaternion}, (1) with a chosen $u \in \{s, t, st\}$.  The argument there shows that given $g, h\in G$, there exists a constant $\a_0(g,h)\in \{\pm 1\}$ such that $P(g)P(h)=\a_0(g,h)P(gh)$.
Hence $g \mapsto P(g)$ defines a degree 2 irreducible projective representation $\tilde \rho_u$ of $G$ whose restriction to $G_{\Q(\sqrt u)}$ is the representation $\sigma_u$.
Let $\tilde \g_u$ be the map sending $gN$ for $g\in N,  ~G_{\Q(\sqrt{u})}\smallsetminus N, ~G_{\Q(\sqrt{v})}\smallsetminus N, ~G_{\Q(\sqrt{vu})}\smallsetminus N$ to $$\begin{pmatrix}
  1&0\\0&1
\end{pmatrix}, ~\begin{pmatrix}
  1&0\\0&-1
\end{pmatrix}, ~\begin{pmatrix}
  0&1\\-1&0
\end{pmatrix}, ~\begin{pmatrix}
  0&1\\1&0
\end{pmatrix},$$ respectively. Note that $\tilde \g_u$ is a degree $2$ {irreducible} projective representation of $G$ trivial on $N$ and \begin{equation}\label{eq:proj-repn}
\rho_\ell=\tilde \rho_u \otimes \tilde \g_u.
\end{equation}
By Tate's vanishing theorem ${\rm H}^2(G, \mathbb C^\times)= 0$ (\cite{tate}), there is a  representation $\g_u : G \rightarrow {\rm GL}_2(\mathbb C)$ with finite image which lifts $\tilde \g_u$. Then the kernel $H$ of $\g_u$
is a normal subgroup of $N$ with quotient $N/H$ finite cyclic. Embed the finite image $\g_u(G)$ into ${\rm GL}_2(\overline {\Q}_\ell)$ and regard $\g_u$ as an $\ell$-adic representation.

For any $g\in G$, there exists $s_u(g)\in   \overline{\Q}_\ell^{\times}$ such that $\tilde \g_u(g)= s_u(g) \g_u (g)$. It follows from $\rho_\ell= \tilde \rho_u \otimes \tilde \g_u= \tilde \rho_u \cdot s_u \otimes s_u^{-1}\cdot \tilde \g_u=\tilde \rho_u \cdot s_u\otimes \g_u $ that $\eta_u:= \tilde \rho_u \cdot s_u$ is also an ordinary representation of $G$. Since $\tilde \rho_u$ restricted to $G_{\Q(\sqrt u)}$ is equal to the representation $\sigma_u$ as noted above, $s_u$ restricted to $G_{\Q(\sqrt u)}$ is a character, denoted by $\chi_u$, with kernel containing $H$. Hence $\eta_u$ is an extension to $G$ of the representation $\sigma_u \otimes \chi_u$ of $G_{\Q(\sqrt u)}$. Note that $\g_u$ and ${\rm{Ind}}_{G_{\Q(\sqrt{u})}}^{G} \chi_u^{-1}$
have the same restrictions to $G_{\Q(\sqrt u)}$, so they differ at most by the quadratic character $\theta_u$ of $G$ with kernel $G_{\Q(\sqrt u)}$. Replacing $\g_u$ by $\g_u \otimes \theta_u$ and $\eta_u$ by $\eta_u \otimes \theta_u$ if necessary, we may assume that
$\g_u = {\rm{Ind}}_{G_{\Q(\sqrt{u})}}^{G} \chi_u^{-1}$ so that
$\rho_\ell \sim \eta_u \otimes {\rm{Ind}}_{G_{\Q(\sqrt{u})}}^{G} \chi_u^{-1}$.  This proves

\begin{theorem}\label{thm:exist-chi_u}
  Let $\rho_\ell$ be a $4$-dimensional $\ell$-adic representation of $G$ with QM over a biquadratic $\Q(\sqrt s, \sqrt t)$.  For $u \in \{s, t, st\}$, let $\sigma_u$ be as in Theorem \ref{quaternion}.  Then there exists a finite character $\chi_u$ of $G_{\Q(\sqrt u)}$ such that $\sigma_u \otimes \chi_u$ extends to a degree $2$ representation $\eta_u$ of $G$ and { $$\rho_\ell \sim {\rm{Ind}}_{G_{\Q(\sqrt{u})}}^{G} \sigma_u \sim \eta_u\otimes{\rm{Ind}}_{G_{\Q(\sqrt{u})}}^{G} \chi_u^{-1}.$$}
\end{theorem}
\smallskip

\begin{remark} As irreducible projective representations of $\Gal(\Q(\sqrt s, \sqrt t)/\Q)$, the $\tilde \g_u$'s are in fact equivalent. This can be shown directly {  (for instance, $\tilde \g_u$ conjugated by $\begin{pmatrix} 1 & i \\ i & 1 \end{pmatrix}$ is equivalent to $ \tilde \g_v$) or seen from the fact that the Schur multiplier of $\Gal(\Q(\sqrt s, \sqrt t)/\Q)$ has only one nontrivial element. Hence the $\eta_u$'s are also projectively equivalent. In particular, they have the same parity, equal to that of $\eta_u$ with $u > 0$.}
\end{remark}

 We close this section by discussing the oddness of the representations occurred.

\begin{prop}\label{oddness} With the same notation and assumption as in Theorem \ref{thm:exist-chi_u}, we have

(a) for positive $u \in \{s, t, st \}$, $\eta_u$ is odd if and only if $\sigma_u$ is odd at both real places of $\Q(\sqrt u)$;

(b) the representation ${\rm{Ind}}_{G_{\Q(\sqrt{u})}}^{G} \chi_u^{-1}$ of $G$ corresponds to an automorphic representation of $\rm{GL}_2$ over $\Q$. Moreover, it is { odd if $\Q(\sqrt s, \sqrt t)$ is not totally real and even otherwise.}
\end{prop}

\begin{proof} Denote by $c$ the complex conjugation in $G$.

(a) The relation $\eta_u|_{G_{\Q(\sqrt u)}}=\sigma_u\otimes\chi_u$ implies $\det (\eta_u) |_{ G_{\Q(\sqrt u)}}= \det (\sigma_u)\cdot \chi_u^2$. As $\det (\eta_u)(g ^{-1}cg )$ remains the same for all $g$ in $G$, we see that $\sigma_u$ and $\eta_u$ have the same parity.

(b) The first statement follows from base change \cite{ac89}. Write $\g_u$ for $ {\rm{Ind}}_{G_{\Q(\sqrt{u})}}^{G} \chi_u^{-1}$ as in the proof of Theorem \ref{thm:exist-chi_u}. Assume that $\Q(\sqrt s, \sqrt t)$ is not totally real so that $c \notin G_{\Q(\sqrt s, \sqrt t)}$. By definition, there is a constant $b \in \overline \Q_\ell^\times$ such that  $ \g_u(c) = \begin{pmatrix}
  b&0\\0&-b
\end{pmatrix}$ or  $ \begin{pmatrix}
  0&b\\b&0
\end{pmatrix}$ or $ \begin{pmatrix}
  0&b\\-b&0
\end{pmatrix}$. In all cases, $\g_u(c)$ has trace zero, hence is odd. When $\Q(\sqrt s, \sqrt t)$ is totally real, $c$ fixes
$\Q(\sqrt s, \sqrt t)$ and hence lies in $G_{\Q(\sqrt s, \sqrt t)}$. Thus $ \g_u(c) = \begin{pmatrix}
  b&0\\0&b
\end{pmatrix}$, which implies that $\g_u$ is even.

\end{proof}

\section{Galois representations attached to noncongruence cusp forms}

 \subsection{Modularity of Scholl representations when $d=1$}\label{SS:d=1}
 Let $\G \subset \t{SL}_2(\Z)$ be a \emph{noncongruence} subgroup, that is, $\G $ is a finite index subgroup of $ \t{SL}_2(\Z)$ not containing any principal congruence subgroup $\G(N)$. For any integer $\kappa \ge 2$, the space $S_\kappa(\G)$ of weight $\kappa$ cusp forms for $\G$ is finite-dimensional;  denote by $d = d(\G, \kappa)$ its dimension. Assume that the compactified modular curve $(\G \backslash\mathfrak H)^*$ is defined over $\Q$  and the cusp at infinity is $\Q$-rational.  For even $\kappa \ge 4$ and any prime $\ell$, in \cite{sch85b}  Scholl constructed an $\ell$-adic Galois representation
 $\rho_\ell: G \rightarrow \t{GL}_{2d}(\Q_\ell)$ attached to $S_\kappa(\G)$. The representations $\rho_\ell$ form a \emph{compatible system} in the sense that there exists a finite set   $S$ of finite primes
of $\Q$ such that for any prime $p \not \in S$ and primes $\ell$ and
$\ell'$ different from  $ p$, the representations $\rho _{\ell}$ and $\rho
_{\ell'}$ are unramified at $p$ and the characteristic polynomials
of $\t{Frob}_p$ under $\rho_{\ell}$ and $\rho_{\ell'}$ have coefficients in
$\Z$ and agree. Scholl also showed that all the roots of the characteristic polynomial of $\t{Frob}_p$
have the same complex absolute value $p^{(\kappa-1)/2}$ (cf. \S5.3 in \cite{sch85b}).
Scholl's results can be extended to odd weights under some
extra hypotheses (e.g., $\pm(\G \cap \G(N))= \pm (\G) \cap \pm
(\G(N))$, where $\pm : \t{SL}_2 (\Z) \to \t{PSL}_2(\Z)$ is the
projection). The readers are referred to the end of \cite{sch85b}
for more details. In this subsection we always
assume that $\rho_\ell$ does exist. Our main concern is
\smallskip

  {\it { Whether $\rho_\ell$ or its dual $\rho_\ell^{\vee}$} is automorphic
as predicted by the Langlands conjecture?}
\smallskip

If so, then {  the associated L-function coincides with the L-function of an automorphic representation of some adelic reductive group. When the reductive group is $\t{GL}_2$ over $\Q$, the representation  is called ``modular" in \S2.}
\smallskip

We provide an affirmative answer to the case $d=1$.

\begin{theorem}\label{cor:d=1}
{When $d=1$,  the  Scholl representation
$\rho_\ell$ of $\G$ is  modular.}
\end{theorem}
\begin{proof} In fact, it has been known for a long time (see \cite{Sconjecture} (4.8)) that the strong form of Serre's conjecture implies the modularity of $2$-dimensional motives over $\Q$, which can be directly applied to our theorem. Here we use a slightly different method which will be useful later.

Since $\{\rho_\ell\}$ forms a compatible system, by Chebotarev Density
Theorem, it suffices to show that there exists an $\ell$ such that
$\rho^\vee_\ell$ is modular. The plan is to conclude this by applying Theorem \ref{tool}
to $\rho_\ell^\vee$ for a large $\ell$. Hence we have to show the
existence of an $\ell$ such that $\rho_\ell^\vee$ satisfies the three hypotheses
in Theorem \ref{tool}. In \cite{sch96}, Scholl proved that, { a general Scholl representation $\rho_\ell$ as above}
 is the  $\ell$-adic realization of a certain motive over
$\Q$  in the sense of
Grothendieck.\footnote[3]{In fact, Scholl's construction is valid for the smallest field $K_\Gamma$ over which the
modular curve $\Gamma \backslash\mathfrak H^*$ is defined. Here we
 restrict ourselves to the case $K_\Gamma = \Q.$} In particular, there exists a smooth projective
$\Q$-scheme $X$ such that $\rho_\ell \simeq \t{H}^{\kappa-1}_{\t{\'et}}(X)$
as $G$-representations. Furthermore, he also proved that the
Hodge type of $\t{H}^{\kappa-1}_{\t{\'et}}(X)$ is $(\kappa -1, 0)^d $ and $ (0, \kappa
-1)^d$ (Theorem 2.12 in \cite{sch85b}). Consequently, we know the
Hodge-Tate weights of $\rho_\ell$ are $0$ and $-(\kappa-1)$ and those of $\rho_\ell^\vee$ are $0$ and $\kappa-1$. Pick $\ell>2\kappa-2$
 so that $X$ has a smooth model over $\Z_{(\ell)}$, the localization of $\Z$ with respect to the prime $\ell$. Then
$\rho_\ell|_{G_\ell}$ is crystalline by Faltings' comparison
theorem (\cite{Faltings}). To complete the proof, it remains to show that $\rho_\ell$ is odd
and absolutely irreducible.

Write $\t{H}^{\kappa-1}(X)$ for $ \t{H}^{\kappa-1}(X, \Q)$,  the singular
cohomology of $X$.  It follows from the comparison theorem between
singular cohomology and \'etale cohomology that $\t{H}^{\kappa-1}(X)
\otimes\Q_\ell \simeq \t{H}^{\kappa-1}_\t{\'et}(X)$ and the isomorphism is
compatible with the action of complex conjugation $c$.
Moreover, the comparison theorem between singular cohomology and de
Rham cohomology implies $ \t{H}^{\kappa-1}(X) \otimes_\Q \C \simeq
\t{H}^{\kappa-1}_\t{dR}(X)$. By Hodge decomposition, we write
$\t{H}^{\kappa-1}_\t{dR}(X)\simeq \bigoplus\limits_{p + q = \kappa-1}
\t{H}^{p, q}(X).$ Since the Hodge type of $\t{H}^{\kappa-1}_\t{dR}(X)$ is
$(\kappa -1, 0)^d $  and $(0, \kappa -1)^d$, we have
$\t{H}^{\kappa-1}_\t{dR}(X)\simeq \t{H}^{0, \kappa-1}(X) \bigoplus \t{H}^{\kappa-1,
0}(X) .$ As $c(\t{H}^{p,q}(X))= \t{H}^{q, p}(X)$ for any $q,
p$,  we conclude that half of the eigenvalues of {$c$} on
$\t{H}^{\kappa-1}_{\t{dR}}(X)$ are $1$ and the other half $-1$. When
$d=1$, this shows that the complex conjugation on the
$2$-dimensional space $\t{H}^{\kappa-1}(X)$ has
 eigenvalues $1$ and  $-1$, hence $\rho_\ell$ is odd.

Finally we prove  that $\rho_\ell$ (hence $\rho_\ell^\vee$) is absolutely
irreducible. Suppose otherwise, namely $\rho_\ell$ is absolutely
reducible. Denote $\rho_\ell\otimes_{\Q_{\ell}} \overline \Q_\ell$ by
$\rho_\ell$ again and its semi-simplification by $\rho _\ell ^\t{ss}$. By
assumption, $\rho ^\t{ss}_\ell \simeq \chi_1 \oplus \chi _2$ for some
Galois characters $\chi_i$.  Consider $\rho_\ell$ restricted to $G_\ell$.
Since $\rho_\ell |_{G_\ell}$ is crystalline with Hodge-Tate weights $\{0,
-(\kappa-1)\}$, reducible and  $\ell >  2{\kappa-2} \ge \kappa$, from the discussion above
Lemma \ref{character}, we see that $\rho_\ell|_{G_\ell}$ is of type
I.
Consequently, we have $$\rho_\ell |_{I_\ell} \sim \begin{pmatrix}  1 & *\\
0& \epsilon^{-(\kappa-1)}_{\ell}\end{pmatrix}.$$ Hence without loss of
generality, we may assume $\chi_1 |_{I_\ell} \simeq \epsilon_\ell ^{-(\kappa-1)}$ and
$\chi_2 |_{I_\ell}\simeq 1$. By Corollary \ref{character}, $\chi_1 \simeq
\epsilon_{\ell}^{-(\kappa-1)} \psi_1$ and $\chi_2 \simeq \psi_2$ where $\psi
_i$ are characters of finite order. Pick a prime $p$ large enough such
that $\rho_\ell$ is unramified at $p$.  Then the eigenvalues of the characteristic
polynomial of $\t{Frob}_p$ under $\rho_\ell$ do not have the same complex absolute values, a contradiction.
 Thus $\rho_\ell$ is absolutely irreducible. This completes the proof of the theorem.

\end{proof}

Contained in the proof above are the following two useful results.

 \begin{cor}\label{cxconj}
Let $\rho_\ell$ be a $2d$-dimensional $\ell$-adic Scholl representation
attached to a space $S_\kappa(\G)$ of dimension $d$. Then the eigenvalues
of the complex conjugation on the space of  $\rho_\ell$ are $1$ and
$-1$, each of multiplicity $d$.
 \end{cor}

\begin{cor}\label{irreducible}
Let $\rho$ be a $2$-dimensional $\ell$-adic representation of  $G$ such that $\rho|_{G_\ell}$
is crystalline with Hodge-Tate weights $\{0, -(\kappa - 1)\}$, where $\kappa \ge 2$. Suppose that $\ell > \kappa$ and
for almost all primes $p$ the eigenvalues of the characteristic polynomial of $\rho(\t{Frob}_p)$ are algebraic numbers with the same complex absolute value. Then $\rho$ is absolutely irreducible.
\end{cor}

\subsection{{Automorphy} of certain Scholl representations with $d\ge 2$}

In the remainder of this paper, we consider the { automorphy}
 of certain Scholl representations for the case $d\ge 2$.

In this subsection, we fix  a degree $d$ cyclic extension $K/\Q$.
Let $\rho_\ell$ be a $2d$-dimensional   subrepresentation of the Scholl representation of $G$ attached to
a  space of weight $\kappa$ cusp forms for a noncongruence
subgroup $\G$. We further assume that the Hodge-Tate weights of $\rho_\ell$ are $\{0, -(\kappa - 1)\}$,
each with multiplicity $d$.  For an $\ell$-adic Galois representation $V$ of $G_K$ and a prime $\gl$ of $\O_K$ above $\ell$, denote by $\t{HT}_\gl (V)$ the set of Hodge-Tate weights of the local Galois representation $V|_{G_{K_\gl}}$ and by $\t{HT}(V)$ the union of $\t{HT}_\gl(V)$ for all primes $\gl$ above $\ell$.

\begin{prop}\label{odd}
(1) Assume that $\rho_\ell \sim
\t{Ind}_{G_K}^{G} \widetilde \rho $ for some 2-dimensional
representation $\widetilde{\rho}$ of $G_K$. Then  $\t{HT}(\widetilde \rho)=\{0, -(\kappa-1)\}$.

(2) Suppose there is a finite character $\chi$ of $G_K$ such that $\widetilde \rho \otimes \chi$
can be {extended} to a degree $2$ representation $\hat \rho$ of $G$, then $\t{HT}(\hat \rho)=\{0, -(\kappa-1)\}$.
\end{prop}

We begin by proving a lemma dealing with Hodge-Tate weights:
\begin{lemma}\label{HTconjugate}Let $E/\Q_\ell$ be a finite Galois extension and $V$ be
a finite dimensional Hodge-Tate representation of $G_E : = \gal
(\overline \Q_\ell/E )$  over $\Q_\ell$. Then for any $\sigma \in \gal(\overline \Q_\ell
/\Q_\ell)$, the Hodge-Tate weights of $V^\sigma$ are the same  as those
of $V$.
\end{lemma}
\begin{proof} Set  $\t{Fil}^iD = (V \otimes_{\Q_\ell} \t{Fil}^iB_{\t{HT}})^{G_E}$ and $\t{Fil}^i D^\sigma = (V^\sigma  \otimes_{\Q_\ell}
\t{Fil}^i B _{\t{HT}})^{G_E}$,  where $B_{\t{HT}}= \bigoplus \limits_{m \in \Z} \C_p (m) $ is the \emph{Hodge-Tate ring}. One can find a construction and discussion of $B_{\t{HT}}$ in \cite{Brian} (see Definition 2.3.6).  Both of them are $E$-vector spaces. It suffices to show that
they have the same dimension over  $E$. For any $\sum_j v_j \otimes a_j
\in \t{Fil}^i D$, where $v_j \in V$ and $a_j \in \t{Fil}^i
B_{\t{HT}}$, it can be easily checked that $\sum_j v_j \otimes
\sigma (a_j) \in \t{Fil}^iD^\sigma$. Hence $\sigma$ induces a
$\Q_\ell$-linear isomorphism between $\t{Fil}^iD$ and $\t{Fil}^i
D^\sigma$. Thus $\t{Fil}^i D $ and $\t{Fil}^i D^\sigma$ have the
same $E$-dimension.
\end{proof}

\begin{proof}[Proof of Proposition \ref{odd}]
Since Hodge-Tate weights are stable when restricted {to} a Galois
subgroup of finite index, the representations $\hat \rho$,
$\widetilde \rho$, and $\widetilde \rho \otimes \chi$ have the same
Hodge-Tate weights. More precisely, suppose that  $V$ is an $\ell$-adic Galois representation of $G$ and $F$ is a finite extension of $\Q$. If $V|_{G_\ell}$ is Hodge-Tate then $\t{HT}_\ell(V|_{G_\ell})= \t{HT}_\gl (V|_{G_{F_\gl}})$ for each prime $\gl$ of $\O_K$ above $\ell$.
Now it suffices to show that $\t{HT}(\widetilde \rho)=  \{0, -(\kappa-1)\}$.
Select $\sigma\in G$ such that its image in $\gal(K/\Q)$ is a
generator.  Note that $\rho_\ell|_{G_K} \simeq
\bigoplus\limits_{i=0}^{d -1}\widetilde \rho^{\sigma^i}$. We claim
that $\t{HT}(\widetilde \rho^{\sigma ^i})$ are the same for
all $i$. If so, since $\rho_\ell$ has Hodge-Tate weights $\{0,
-(\kappa-1)\}$ and each weight appears with multiplicity $d$,  we get
that $\t{HT}(\widetilde \rho)=\{0, -(\kappa-1)\}$.  To
prove the claim, it suffices to show that for any prime $\mathfrak
l$ of $\O_K$ we have $\t{HT}_\gl(\widetilde \rho)= \t{HT}_{\sigma ^i (\gl)}(\widetilde \rho ^{\sigma
^i})$. We should be careful that we can not directly use Lemma \ref{HTconjugate} here because it is not clear that $\sigma^i(\gl) = \gl.$

Observe that $\widetilde
\rho|_{G_{K_\gl}}$ is isomorphic to $\widetilde \rho^{\sigma^i}$
restricted to $\sigma ^i G_{K_\gl}\sigma^{-i}$. Let $\varpi$ be the
maximal ideal  of $\O_{\overline \Q}$ above $\gl$ which gives rise to the decomposition group
 $G_{K_\gl}$. Then there exists $\tau \in G$ such that
$G_{K_{\sigma^i (\gl)}} = \tau G_{K_\gl}\tau^{-1}$. Note that
$\sigma^i(\gl) = \tau (\gl) $, so there exists $\lambda \in G_K$
such that $\lambda \sigma^i  (\varpi)= \tau (\varpi)$. Since
$\lambda$ is in $G_K$, it is easy to see that $\widetilde
\rho^{\sigma^i}$ restricted to $\sigma ^i G_{K_\gl}\sigma^{-i}$ is
isomorphic to $\widetilde\rho ^{\lambda \sigma ^i}$ restricted to $\lambda
\sigma^i G_{K_\gl}(\lambda \sigma ^i)^{-1}$ resulting from the isomorphism
$\sigma ^i G_{K_\gl}\sigma^{-i} \simeq \lambda \sigma^i G_{K_\gl}(\lambda
\sigma ^i)^{-1} $ given by conjugation. So without loss of generality, we
may assume that $\sigma^i (\varpi) = \tau (\varpi).$ Thus $\sigma^i
\tau ^{-1} $ is in $ G_\ell$ but may not be in $G_{K_\gl}$.

Identify $\sigma^i G_{K_\gl} \sigma ^{-i}$  with
$G_{K_{\sigma^i(\gl)}}= \tau G_{K_\gl} \tau^{-1}$  via conjugation by
$\sigma^i \tau ^{-1}$. The  identity map $\t{Id}:\widetilde\rho ^{\sigma^i}
\to\widetilde \rho^{\tau} $ is  an isomorphism of $G_K$-modules.
Now we need to show that
$\widetilde\rho ^{ \tau }$ and $\widetilde\rho ^{\sigma ^i}$ have the same Hodge-Tate
weights  on $G_{K_{\sigma ^i (\gl)}}$.  This follows from  Lemma
\ref{HTconjugate} and the fact that  $\widetilde\rho ^{\tau} =
(\widetilde\rho^{\sigma ^i}) ^{\tau \sigma ^{-i}  }$.
\end{proof}

\medskip

We are ready to prove the {automorphy} of certain Scholl representations. First suppose that the Scholl representation $\rho_\ell$ has degree $4$ and admits QM over a quadratic field $\Q(\sqrt s)$. By Theorem \ref{quaternion}, (2), we have, over $\Q_\ell(\sqrt {-1})$, $$\rho_\ell = \sigma \oplus (\sigma \otimes \theta) = \sigma \otimes
{\rm{Ind}^G_{G_\Q(\sqrt s)} 1} = \rm{Ind}^G_{G_\Q(\sqrt s)} \sigma|_{G_\Q(\sqrt s)}$$
for a degree $2$ representation $\sigma$ of $G$ and a quadratic character $\theta$ of $G$ with kernel $G_{\Q(\sqrt s)}$.
So $\sigma$ and $\sigma \otimes \theta$ restricted to the decomposition group $G_\ell$ are crystalline, and they have the same
Hodge-Tate weights $\{0, -(1 - \kappa)\}$ by Proposition \ref{odd}. We conclude their absolute irreducibility from Corollary \ref{irreducible}. Proposition \ref{cxconj} says that $\rho_\ell$ has trace zero at the complex conjugation, hence by Proposition \ref{conjugation}, both $\sigma$ and $\sigma \otimes \theta$ are odd. It then follows from Theorem \ref{tool} that $\sigma$ and $\sigma \otimes \theta$ are modular for large $\ell$, and thus for all $\ell$ by compatibility. We record this in

\begin{theorem}\label{QMoverquadratic}
Suppose that the Scholl representation $\rho_\ell$ has degree $4$ and admits QM over a quadratic field $\Q(\sqrt s)$. Then over $\Q_\ell (\sqrt {-1})$ it decomposes as a direct sum of two modular representations $\sigma$ and $\sigma \otimes \theta$.  Here $\theta$ is the quadratic character of $\gal(\Q(\sqrt s)/\Q)$.
\end{theorem}

Next we consider a more general situation.

\begin{theorem}\label{thm:3.2.1}
Let $K$ be a degree $d$ cyclic extension of $\Q$
and $\rho_\ell$ be a $2d$-dimensional
$\ell$-adic subrepresentation of a Scholl representation of $G$ as above. Assume that
 \begin{itemize}

 \item[(a)] $\rho_\ell$ is induced from a 2-dimensional representation
  $\widetilde{\rho}$ of $G_K$;
 \item[(b)] There exists a character $\chi$ of
  $G_K$  of  finite order such that $\widetilde{\rho}\otimes \chi$ can be extended to an $\ell$-adic representation $\hat \rho$ of $G$;
\item[(c)] $\rho_\ell|_{G_\ell}$ is crystalline and $\chi$ is unramified at any prime above $\ell$.

 \end{itemize}
Then $\hat \rho$ is  absolutely irreducible.  If we further assume that
  \begin{itemize}
   \item[(d)] $K$ is unramified over $\ell$ and $\ell > 2\kappa-2$;

  \item[(e)] $\hat \rho$ is odd,
   \end{itemize}
then the dual of $\hat \rho$ is isomorphic to a modular representation $\rho_{g, \lambda}$ of $G$  attached to a weight $\kappa$ cuspidal newform $g$, and
the following relations on L-series hold:
$$L(s,\rho_\ell^{\vee})=L(s, {\widetilde{\rho}} ^{\vee}) = L(s, (\rho_{g, \lambda}|_{G_K})\otimes
\chi).$$
\end{theorem}\bigskip

\begin{remark} The character $\chi$ of $G_K$
corresponds to an automorphic form $\pi_{\chi}$ of
$\t{GL}_d(\BA_{\Q})$ by base change \cite{ac89}. The above expression shows that  the
semi-simplification of $\rho^\vee _\ell$ is automorphic, arising from
the form $g \times \pi_{\chi}$ of the group $\t{GL}_2(\BA_{\Q})
\times \t{GL}_d(\BA_{\Q})$.
This is because  $$(\rho_\ell^\vee)^\t{ss} = (\t{Ind}_{G_K}
^{G} (\widetilde \rho^\vee)^\t{ss})^\t{ss}  = (\t{Ind}_{G_K}^{G} (\rho_{g, \lambda}|_{G_K}
\otimes \chi ))^\t{ss},$$ and as representations of $G$,
$\t{Ind}_{G_K}^{G} (\rho_{g, \lambda}|_{G_K} \otimes \chi )$ is
isomorphic to  $ \rho _{g, \lambda} \otimes \t{Ind}_{G_K}^G \chi$
(cf. Theorem 38.5 \cite[pp. 268]{curtis-reiner}).

When $d=2$, it follows from the work of Ramakrishnan \cite{Ramakrishnan00} that an automorphic form for
$\t{GL}_2(\BA_{\Q})
\times \t{GL}_2(\BA_{\Q})$ corresponds to an automorphic form for $\t{GL}_4(\BA_{\Q})$. In this case the $L$-function attached to the degree $4$ representation $\rho_\ell^{\vee}$ is an automorphic $L$-function for $\t{GL}_4(\BA_{\Q})$. { This fact also holds for general $d$ by base change \cite[\S3.6]{ac89}. More precisely, that $\hat \rho^\vee$ comes from an automorphic form of $\t{GL}_2$ over $\Q$ implies that $\tilde \rho^\vee$ comes from an automorphic form of $\t{GL}_2$ over $K$, which in turn implies that $\rho_\ell^\vee$ corresponds to an automorphic form of $\t{GL}_{2d}$ over $\Q$.}

\end{remark}

\begin{proof}[Proof of Theorem \ref{thm:3.2.1}]  We first show that $\hat \rho  $ is absolutely irreducible.  Suppose otherwise.
 Enlarging the field of coefficients if necessary, we may assume that $
(\hat \rho)^\t{ss}  = \eta _1 \oplus \eta_2$ for some characters $\eta_1$ and $\eta_2$ of $G$.
Note that $(\hat \rho |_{G_K})^\t{ss} = (\widetilde \rho \otimes \chi)^\t{ss}$.
Since $\rho_\ell$ is crystalline at $\ell$,
$\widetilde \rho$ is crystalline at any prime above $\ell$, which in turn implies that $\widetilde \rho
\otimes \chi$ is crystalline at any prime above $\ell$ because
$\chi$ is unramified at any prime dividing $\ell$.
Consequently,
each $\eta_i$ is crystalline at $\ell$. Hence $\eta_i = \mu_i \epsilon _\ell^{\kappa_i}$ with $\mu_i$ a finite character  and $\kappa_i$ the Hodge-Tate weight of $\eta_i$ by Corollary \ref{character2}. Since $ \hat \rho$ has Hodge-Tate weights $0$ and $-(\kappa-1)$ by Proposition \ref{odd}, we see that each
$\eta_i$ has Hodge-Tate weight either $0$ or $-(\kappa-1)$. Therefore $\widetilde \rho^\t{ss} =
(\mu_1 \epsilon_1 \oplus \mu_2 \epsilon_\ell ^{-(\kappa-1)})|_{G_K}$. This contradicts the fact
that the roots of the characteristic polynomial of $\t{Frob}_p$ under the representation
$\rho_\ell =\t{Ind}^G_{G_K}\widetilde \rho$  have the same complex
absolute value. So $\hat \rho$ cannot be absolutely reducible.

Now we prove the remaining assertions under the additional assumptions.  As shown above, $\hat \rho$ is absolutely irreducible and has Hodge-Tate
weights $\{0, -(\kappa-1)\}$. Further, since $K $ is
unramified above $\ell$, $\hat \rho|_{G_\ell}$ is crystalline.  Finally, since $\hat \rho$ is odd by assumption, we conclude from Theorem
\ref{tool}  that $\hat \rho^\vee$ comes from a  weight $\kappa$ cuspidal newform $g$ as described
  in Theorem \ref{Deligne}. Hence $\widetilde \rho = \rho^\vee_{g, \lambda}
|_{G_K} \otimes \chi ^{-1}$ and the relations on L-functions
hold. \end{proof}

\begin{remark}\label{rem:3.2.4}
Since we have a compatible family of Scholl representations $\rho_\ell$ constructed from geometry,
there always exists a prime $\ell$ large enough such that the conditions (c) and (d) are satisfied.
However, the oddness of $\hat \rho$ is not automatic.
\end{remark}

\subsection{Degree $4$ Scholl representations with QM and Atkin-Swinnerton-Dyer conjecture} Hecke operators played a fundamental role in the arithmetic of
congruence modular forms.  It was shown in \cite{tho89, berger94,
sch97} that similarly defined Hecke operators yielded little
information on genuine noncongruence forms. When the modular curve of a noncongruence
subgroup $\G$ has a model over $\Q$, Atkin and Swinnerton-Dyer in
\cite{a-sd} predicted a ``$p$-adic'' Hecke theory in terms of
$3$-term congruence relations on Fourier coefficients of  weight ${\kappa}$ cusp forms in $S_{\kappa}(\G)$ as follows. Suppose the cusp at $\infty$ is a $\Q$-rational point with
cusp width $\mu$. The $d$-dimensional
space $S_{\kappa}(\G)$ admits a basis whose Fourier coefficients are in a finite extension of $\Q$; moreover, there is an integer $M$ such that for each prime $p \nmid M$, these Fourier coefficients are $p$-adically integral, that is, integral over $\Z_p$. Atkin and Swinnerton-Dyer
conjectured in \cite{a-sd} that, given ${\kappa} \ge 2$ even, for almost all primes $p$,
 $S_{\kappa}(\G)$ has a $p$-adically integral basis $f_j(z) = \sum_{n \ge 1} a_j(n) q^{n/\mu}$, where $q = e^{2 \pi i z}$ and  $1 \le
j \le d$, depending on $p$,  whose Fourier coefficients satisfy the
$3$-term congruences
\begin{equation}\label{eq:ASD}
 a_j(pn)-A_j(p)a_j(n)+ p^{({\kappa}-1)/2} a_j(n/p)\equiv 0 \mod p^
  {({\kappa}-1)(1+\ord_p n)} \quad \forall n\ge 1
\end{equation}
for some algebraic integers $A_j(p)$  satisfying the Ramanujan bound. The congruence (\ref{eq:ASD}) is carefully explained in \cite{a-sd} and \cite{sch85b}. It can be summed up as the quotient
$(a_j(pn)-A_j(p)a_j(n)+ p^{({\kappa}-1)/2} a_j(n/p))/p^{({\kappa}-1)(1+\ord_p n)}$ { is integral over $\Z_p$, as in \cite{all05}. In what follows, this is the interpretation of ASD congruences we shall adopt.}
  The ${\kappa}=2$ case of this conjecture was
verified in \cite{a-sd} for $d=1$ and  discussed in \cite{car71} using formal group language.

Scholl showed in \cite{sch85b} that, for ${\kappa} \ge 4$ even and
 $p$ large enough, all $p$-adically integral cusp forms $f = \sum_{n \ge 1}
a(n) q^{n/\mu}$ in
$S_{\kappa}(\G)$  satisfy $(2d+1)$-term congruences
\begin{equation*}
  a(p^dn)+C_1(p)a(p^{d-1}n)+ \cdots + C_{2d}(p)a(n/p^d) \equiv 0 \mod
  p^{({\kappa}-1)(1+\ord_p n)} \quad \forall n\ge 1,
\end{equation*}
in which $H_p(x):= x^{2d} + C_1(p)x^{2d-1} + \cdots + C_{2d}(p) \in \Z[x]$ is the
characteristic polynomial of the associated Scholl representation $\rho_{\ell}$  at $\text{Frob}_p$.
Similar statement holds for ${\kappa}$ odd. Hence the ASD conjecture holds for $d=1$. For arbitrary $d$, proving the ASD conjecture amounts to factoring $H_p(x)$ as a product
of $d$ quadratic polynomials $x^2 - A_j(p)x + B_j(p)$, $1 \le j \le d$, and finding a basis $f_j =\sum_{n \ge 1} a_j(n) q^{n/\mu}$ of $S_{\kappa}(\G)$ with  $p$-adically integral coefficients such that the $3$-term congruences
\begin{equation*}
  a_j(pn)-A_j(p)a_j(n)+B_j(p)a_j(n/p) \equiv 0 \mod p^
  {({\kappa}-1)(1+\ord_p n)} \quad \forall n\ge 1
\end{equation*} hold for $1 \le j \le d$.
\medskip

Assume that $S_{\kappa}(\G)$ has a $2$-dimensional subspace $S$ to which one can associate a family of compatible degree-$4$ Scholl representations $\{\rho_\ell\}$ of $G$  over $\Q_\ell$. Suppose  that there is a finite extension $F$ of $\Q$ and there are operators $J_s$ and $J_t$ acting on
 extensions over  fields in $\Q_\ell \otimes F$ of the representation spaces of all $\rho_\ell$ that satisfy the QM conditions in Definition \ref{QM} over a biquadratic $\Q(\sqrt s, \sqrt t)$. Then, as stated in Corollary \ref{factorization}, for almost all primes $p$ they give rise to a factorization of  the characteristic polynomial $H_p(x)$ of $\rho_\ell({ \t{Frob}}_p)$ into a product of two quadratic polynomials $H_{\gp_{\pm}, u}(x) = x^2 - A_{\pm, u}(p)x + B_{u}(p)$ by choosing a $u \in \{s, t, st\}$ such that $p$ splits  into two places $\gp_{\pm}$ in $\Q(\sqrt u)$. Recall that $H_{\gp_{\pm}, u}(x)$ are the respective characteristic polynomial of $\sigma_u({ \t{Frob}}_{\gp_+})$ and $\sigma_u({ \t{Frob}}_{\gp_-}) = (\sigma_u \otimes \delta_u)({ \t{Frob}}_{\gp_+})$, where $\sigma_u$ and $\delta_u$ are as in Theorem \ref{quaternion}.

For $p \nmid M$ and unramified in $\Q(\sqrt s, \sqrt t)$, let $V_p$ be the $4$-dimensional $p$-adic Scholl space containing the space $S_p$ of  $p$-adically integral forms in $S$ as a  subspace and its dual $(S_p)^\vee$ as a quotient (cf. \cite{sch85b}). On $V_p$ there is the action of $F_p$, the Frobenius at $p$. Scholl proved in \cite{sch85b} that, for $\ell \ne p$ and $\rho_\ell$ unramified at $p$, $\rho_\ell({ \t{Frob}}_p)$ and $F_p$ have the same characteristic polynomial.

 Assume that  the operators $J_s$ and $J_t$ also act on $S$, preserving forms whose Fourier coefficients are in a suitable number field $K$ which are $p$-adically integral
for almost all $p$. For such a $p$, their actions on $S_p$ extend to $V_p$, again denoted by $J_s$ and $J_t$. Suppose that they satisfy
the condition (a) in Definition \ref{QM} and

(b)' for $u \in \{s, t, st \}$, the operator $J_u$ commutes with $F_p$ for $p$ split in $\Q(\sqrt u)$, and $J_uF_p=-F_pJ_u$ for $p$ inert in $\Q(\sqrt u)$.  Here $J_{st} = J_sJ_t$.

\noindent We say that $V_p$ admits QM over $\Q(\sqrt s, \sqrt t)$ for simplicity.
The same argument as in the proof of Theorem \ref{quaternion} shows that over  any field in $\Q_p \otimes K(\sqrt {-1})$, for every $u \in \{s, t, st \}$, each $\pm i$-eigenspace $V_{p, \pm, u}$ of $J_u$ is $2$-dimensional, containing a form in $S_p$,
and $J_v$, where $v \in \{s, t, st\}$ and $ v \ne u$, permutes the two eigenspaces.
If $p$ splits completely in $\Q(\sqrt s, \sqrt t)$, then  $V_{p, \pm, u}$ are invariant under $F_p$ and $J_v$ commutes with  $F_p$. Therefore the actions of $F_p$ on $V_{p, \pm, u}$ are isomorphic and hence have the same characteristic polynomial, equal to $H_{\gp_{\pm}, u}(x)$.
If $p$ does not split in $\Q(\sqrt s, \sqrt t)$, choose an element $u \in \{s, t, st\}$ such that $p$ splits in $\Q(\sqrt u)$. Then $V_{p, \pm, u}$ are invariant under $F_p$. On the $\ell$-adic side, the representation space $W_\ell$ of $\rho_\ell$ decomposes similarly into two $\pm i$-eigenspaces $W_{\ell,p, \pm, u}$ of $J_u$, each is invariant under $\rho_\ell({ \t{Frob}}_p)$. Moreover, viewed as $J_u\rho_\ell({ \t{Frob}}_p)$ invariant spaces, they are intertwined by $J_v$ because of condition (b) in Definition \ref{QM}. Therefore
$J_u\rho_\ell({ \t{Frob}}_p)$ has the same characteristic polynomial on $W_{\ell,p, \pm, u}$.

A normalizer of $\G$ acts on the space $S$ and hence $V_p$. Further, its action on the modular curve $X_\G$ induces an action on the $\ell$-adic cohomology group, which is the representation space of $\rho_\ell$. When $J_s$ and $J_t$  arise from normalizers of $\G$, or  algebraic real linear combinations of such, we conclude from the argument in \cite[Prop. 4.4 and proof]{sch85b} that $J_uF_p$ on $V_p$ and $J_u \rho_\ell({ \t{Frob}}_p)$ on
$W_\ell$ have the same characteristic polynomials.  Combined with the fact that $\rho_\ell({ \t{Frob}}_p)$ and $F_p$ have the same characteristic polynomials, we get, by using the argument of Lemma 8 of \cite{all05}, that when $p$ splits in $\Q(\sqrt u)$,
 $F_p$ on $V_{p, \pm, u}$ and $\rho_\ell({ \t{Frob}}_p)$ on $W_{\ell,p, \pm, u}$ have the same characteristic polynomials $H_{\gp_{\pm}, u}(x) = x^2 - A_{\pm,u}(p)x+B_{u}(p)$.

As observed above, for almost all primes $p$, $V_{p, \pm, u}$
contain a nonzero $f_{\pm, u} = \sum_{n \ge 1} a_{\pm, u}(n) q^{n/\mu} \in S$ with $p$-adically integral Fourier coefficients.  They form a basis of $S$. It follows from  \cite{sch85b, L5} that the ASD congruences hold, namely,
\begin{eqnarray}\label{eq:ASD2}
 a_{\pm, u}(pn)-A_{\pm, u}(p)a_{\pm, u}(n)+B_{ u}(p)a_{\pm, u}(n/p) \equiv 0 \mod p^
  {({\kappa}-1)(1+\ord_p n)} \quad \forall n\ge 1.
\end{eqnarray}
\smallskip

Now we turn to the modularity of the degree $4$ representations $\rho_\ell$ with QM over $\Q(\sqrt s, \sqrt t)$ as above.  Theorem \ref{thm:exist-chi_u} says for each $u \in \{s, t, st\}$, there is a finite $\ell$-adic characters $\chi_u$ of $G_{\Q(\sqrt u)}$ such that $\sigma_u \otimes \chi_u$ { extends to a representation $\eta_u$ of $G$.}  We know from Corollary 4.1.2 that all $\rho_\ell$ have trace $0$ at the complex conjugation. If there is a large $\ell$ such that the representation $\rho_\ell$ is over a field
not containing a square root of $-1$, then by combining Propositions \ref{conjugation} and \ref{oddness}, we get that $\eta_u$ is odd { for this} and hence all $\ell$ because of compatibility. We then conclude from Theorem \ref{thm:3.2.1} that the dual of $\rho_\ell$ is automorphic.
In this case the dual of $\sigma_u$ corresponds to an automorphic form $h_u$ of $\t{GL}_2$ over  $\Q(\sqrt u)$. As $\sigma_s$, $\sigma_t$, $\sigma_{st}$ have the same restrictions to $G_{\Q(\sqrt s, \sqrt t)}$, the
 three automorphic forms $h_u$ of $\t{GL}_2$ over  $\Q(\sqrt u)$ for $u \in \{s, t, st\}$ base change to the same automorphic form $h$ of $\t{GL}_2 $ over $\Q(\sqrt s, \sqrt t)$.

We summarize the above discussions in the theorems below.

 \begin{theorem}\label{QMdescent} Let $\{\rho_\ell\}$ be a family of compatible degree-$4$ Scholl representations of $G$ associated to a $2$-dimensional subspace $S$ of $S_{\kappa}(\G)$ for a noncongruence subgroup $\G$. Suppose  that
there exists a finite real extension $F$ of $\Q$ and   operators $J_s$ and $J_t$ acting on the spaces of $\rho_\ell$ extended over fields in $\Q_\ell \otimes F$ satisfying the QM conditions in Definition \ref{QM} over a biquadratic field $\Q(\sqrt s, \sqrt t)$.
Then for each $u \in \{s, t, st\}$, there is an automorphic form $h_u$ of $\t{GL}_2(\mathbb A_{\Q(\sqrt u)})$ such that the following $L$-functions have the same local factors over all primes of $\mathbb Z$:
$$ L(s,  \rho_\ell^\vee) = L(s, h_s) = L(s, h_t) = L(s, h_{st}).$$
\end{theorem}

In addition to these three automorphic interpretations, $L(s, \rho_\ell^\vee)$ also agrees with the $L$-functions of an automorphic form of $\t{GL}_2(\mathbb A_\Q) \times \t{GL}_2(\mathbb A_\Q)$ and an automorphic form of $\t{GL}_4(\mathbb A_\Q)$, as remarked after Theorem \ref{thm:3.2.1}.

If $\rho_\ell$ is unramified at a prime $p$ which splits in $\Q(\sqrt u)$, then $h_u$ is an eigenfunction of the Hecke operators at the two places $\gp_{\pm}$ of $\Q(\sqrt u)$ above $p$, and the product of the local $L$-factors attached to $h_u$ at these two places is
\begin{eqnarray*}
\frac{1}{(1 - A_{+,u}(p)p^{-s} + B_{u}(p)p^{-2s})(1 - A_{-,u}(p)p^{-s} + B_{u}(p)p^{-2s})},
\end{eqnarray*}
where $H_{\gp_{\pm}, u}(x) = x^2 - A_{\pm,u}(p)x+B_{u}(p)$ are the characteristic polynomials of $\sigma_u({ \t{Frob}}_{\gp_{\pm}})$.

\begin{theorem}\label{ASD}  Keep the same notation and hypotheses as in Theorem \ref{QMdescent}. Assume the modular curve $X_\G$ has a model defined over $\Q$ such that the cusp at $\infty$ is a $\Q$-rational point  with cusp width $\mu$. Further, assume that $J_s$ and $J_t$ arise from  algebraic real linear combinations of normalizers of $\G$ whose actions on $V_p$ admit QM over $\Q(\sqrt s, \sqrt t)$  for almost all $p$.
 Then there exists a  finite set of primes $T$, including ramified primes, primes $< 2{\kappa}-2$, and primes where $X_\G$ has bad reductions, such that for each $p \notin T$,
 $S$  has a $p$-adically integral basis $f_{\pm, p} = \sum_{n \ge 1} a_{\pm, u}(n) q^{n/\mu}$ which are $\pm i$-eigenfunctions of $J_u$ for some $u \in \{s, t, st\}$ such that $p$ splits in $\Q(\sqrt u)$. The Atkin and Swinnerton-Dyer congruences (\ref{eq:ASD2}) hold with $A_{\pm, u}(p)$ and $B_{ u}(p)$ coming from the characteristic polynomials $H_{\gp_{\pm}, u}(x) = x^2 - A_{\pm,u}(p)x+B_{u}(p)$ of $\rho_\ell({ \t{Frob}}_p)$ on the $\pm i$-eigenspaces of $J_u$ on the space of $\rho_\ell$. Moreover, $(1 - A_{+,u}(p)p^{-s} + B_{u}(p)p^{-2s})^{-1}$ and $(1 - A_{-,u}(p)p^{-s} + B_{u}(p)p^{-2s})^{-1}$ are the local factors
 at the two places $\gp_{\pm}$ of  $\Q(\sqrt u)$ above $p$ of an automorphic form $h_u$ for $\t{GL}_2 $ over $\Q(\sqrt u)$.
 \end{theorem}

Observe that the basis for which the ASD conjecture holds depends on $p$ modulo the discriminant of $\Q(\sqrt s, \sqrt t)$. Further, the ASD congruences (\ref{eq:ASD2}) describe congruence relations between Fourier coefficients of noncongruence forms $f_{\pm, p}$ and those of congruence forms $h_u$.

\begin{remark}  Under the same assumptions as above except that the field $\Q(\sqrt s, \sqrt t)$ is a quadratic extension of $\Q$, the same argument shows that, for almost all primes $p$, if each eigenspace of $J_s$ contains a nonzero form in $S$, then the ASD conjecture on $S$ holds at such $p$.
\end{remark}

\section{Applications}

As applications, we exhibit some examples of Scholl representations which admit QM. We use Theorem  \ref{QMdescent}  to conclude the automorphy of the representation and Theorem \ref{ASD} to conclude the ASD congruences. We expect that there is a wide variety of cases to which Theorem \ref{thm:3.2.1} could be applied.

\subsection{Old cases}\label{sec:3.3} We first re-establish  several known automorphy results, previously proved using Faltings-Serre modularity
criterion. While the new argument is more conceptual, no information on the congruence modular form  giving rise to the Galois representation is revealed.

In a series of papers \cite{lly05, all05, long061}, a sequence of genus
0 normal subgroups, denoted by $\G_n$, of $\G^1(5)$ are considered.
Their modular curves are $n$-fold covers of the modular curve for
$\G^1(5)$, ramified only at two cusps $\infty$ and $-2$ of
$\G^1(5)$ with ramification degree $n$. The group $\G_n$ is
noncongruence when $n\neq 1,5$, its modular curve is defined over $\Q$ with the cusp at $\infty$ a $\Q$-rational point, and
an explicit basis for
$S_3(\G_n)$ with rational coefficients was constructed (cf.  \cite[Prop. 1]{all05}).
To the $(n-1)$-dimensional space $S_3(\G_n)$ Scholl has attached  a  compatible family of $2(n-1)$-dimensional $\ell$-adic Galois representations of $G$.
The functions in $S_3(\G_n)$ belonging to a nontrivial
supergroup $\G_m$ of $\G_n$ are called ``old" forms; denote by
$S_3(\G_n)^{\text{new}}$ the subspace of forms in $S_3(\G_n)$
orthogonal to the ``old" forms. Accordingly, there is a compatible
family of degree $2\varphi(n)$-dimensional $\ell$-adic representations
$\rho_{3,\gl,n}^{\text{new}}$ of $G$ attached to
$S_3(\G_n)^{\text{new}}$. Here $ \varphi(n)$ stands for the Euler
phi-function (cf. \cite[\S 3]{long061}).  For almost all $\ell$, $\rho_{3,\gl,n}^{\text{new}}|_{G_\ell}$ is crystalline with Hodge Tate weights $\{0, -2\}$, each with multiplicity $\varphi(n)$, and the action of complex conjugation has trace zero. Let $A=\begin{pmatrix}
-2&-5\\1&-2
\end{pmatrix}$ and $\zeta=\begin{pmatrix}
  1&5\\0&1
\end{pmatrix}$; they induce operators $A^*$ and $\zeta^*$ on the
cohomology level, whose actions on the space of $\rho_{3,\gl,n}^{\text{new}}$
 satisfy the following relations:
 $$(A^*)^2=-I, \quad
(\zeta^*)^n=I, \quad \text{and} \quad \zeta^* A^* \zeta^* =A^*.$$

\begin{theorem}[\cite{lly05, all05, long061}]\label{thm:old}
When $n=3,4,6,$ the degree $4$ representations $\rho_{3,\gl,n}^{\text{new}}$ are
automorphic and the ASD conjecture on $S_3(\G_n)^{\text{new}}$ holds.
\end{theorem}

\begin{proof}
 When $n=3$, $A^*$ and $\zeta^*$ are both defined over $\Q(\sqrt {-3})$.
The operators $J_{-3} := A^*$ and $J_{-27} := \frac{1}{\sqrt 3}(\zeta^* - (\zeta^*)^{-1})$
define QM over $\Q(\sqrt {-3})$ on the representation space of $\rho_{3,\gl,3}^{\text{new}}$.
The automorphy of  this representation follows from Theorem \ref{QMoverquadratic}.

When $n=4$, the $4$-dimensional space of
$\rho_{3,\gl,4}^{\text{new}}$ admits QM  over $\Q(\sqrt{-1}, \sqrt{-2})$ by $J_{-1}=A^*$ and
$J_{-2}=\frac 1{\sqrt 2}A^*(1+\zeta^*)$. We can use $F=\Q(\sqrt 2)$ in Theorem \ref{QMdescent}.
Similarly when $n=6$, the space of $\rho_{3,\gl,6}^{\text{new}}$ admits QM  over $\Q(\sqrt{-1}, \sqrt{-3})$ by
$J_{-1}=A^*$ and $ J_{-3}=\frac 1{\sqrt 3}(\zeta^*-(\zeta^*)^{-1})$. In this case, $F=\Q(\sqrt 3)$.
 By Theorem \ref{QMdescent}
$\rho_{3,\gl,n}^{\text{new}}$ is automorphic for $n = 4$ and $6$.

Moreover, the ASD conjecture on $S_3(\G_n)^{\text{new}}$ holds for $n = 3, 4, 6$ as it is easy to verify that the assumptions in Theorem \ref{ASD} and the remark following it are satisfied.
\end{proof}

\begin{remark}
The above automorphy argument does not work for $n>6$. However, we conjecture that when $n=p$ is an odd prime, the
  representation $\rho_{3,\gl,p}$ is related to a Hilbert modular form.
\end{remark}

Another known example of 4-dimensional Scholl representation with QM by a biquadratic field has been investigated in \cite{HLV10}.

\subsection{A new example  with  QM over a biquadratic field}

The group $\G^1(6)$ is normalized by $W_6$, $W_2=\begin{pmatrix}
 2&-6\\1&-2
\end{pmatrix}$ and $W_3=\begin{pmatrix}
  3&6\\1&3
\end{pmatrix}$. It is a torsion-free index-12 genus $0$ subgroup of $P\SL$ with four cusps:
$\infty,$ $0=W_6^{-1}\cdot \infty,$ $-3=W_3^{-1}\cdot \infty,$ and
$-2={W_2}^{-1}\cdot \infty$. Its space of weight 3 cusp forms is
0-dimensional. The function $\displaystyle
F:=\frac{\eta(z)^4\eta(2z)\eta(6z)^5}{\eta(3z)^4}$ is a weight 3
Eisenstein series for $\G^1(6)$ which vanishes at every cusp except
$-2$.

Let $B=\frac{\eta(2z)^3\eta(3z)^9}{\eta(z)^3\eta(6z)^9}$ and $\G$ be
an index-6 subgroup of $\G^1(6)$ whose  modular curve is a 6-fold
cover of the modular curve for $\G^1(6)$,  ramified totally at the two
cusps $\infty$ and $-2 = W_2^{-1}\infty$ and unramified elsewhere.
I.e., $\G$ admits $\displaystyle (B)^{1/6}$ as a Hauptmodul. The
space $S_3(\G)$ has a basis $\langle F_j=B^{(6-j)/6}\cdot F
\rangle_{j=1}^5$. The functions $F_j$ with $2 \le j \le 4$ are for
certain supergroups of $\G$, thus are regarded as ``old" forms.  The
``new" space $S_3(\G)^{\text{new}}$   for $\G$ is spanned by
$$F_1=(B)^{5/6}\cdot
F  \quad \text{and}  \quad  F_5=(B)^{1/6}\cdot F.$$

On the $p$-adic side, the action of $W_2$ on $S_3(\G)^{\t{new}}$ is as
follows. By noting $B|W_2=\frac{-8}{B}$, we claim
\begin{equation}
F_1 | W_2 = 2 e^{2\pi i 5/12}F_5   \quad {\rm and} \quad F_5 |
W_2 = \frac{1}{2} e^{2 \pi i / 12} F_1.
\end{equation}

\noindent Indeed, by definition, $B = c \frac{(F | W_2)}{F}$
for some nonzero constant $c$,
$F_1 = B^{5/6} F = c^{5/6} (F | W_2)^{5/6} F^{1/6}$, and $F_5 =
B^{1/6} F = c^{1/6} F^{5/6} (F | W_2)^{1/6}$. Since $(W_2)^2 =
\begin{pmatrix} -2 & 0\\ 0 & -2 \end{pmatrix}$ and $F$ has weight $3$,
we have $F | W_2 W_2 = -F$. Therefore

$$F_1 | W_2 = c^{5/6} (-1)^{5/6} F^{5/6} (F | W_2)^{1/6} = c^{2/3}
(-1)^{5/6} F_5$$
\noindent and
$$F_5 | W_2 = c^{1/6} (F | W_2)^{5/6} (-1)^{1/6} F^{1/6} = c^{-2/3}
(-1)^{1/6} F_1.$$ To compute the constant $c$, we use the fact that
$B|W_2 = \frac{-8}{B}$. This gives $c^2 = 8$ and hence $c^{2/3} =
2$.

Let $\zeta=\begin{pmatrix} 1&1\\0&1
\end{pmatrix}$. It is easy to verify that
\begin{equation}
  F_1|{\zeta}=e^{\pi i /3} F_1 \quad {\rm and} \quad F_5|{\zeta}=e^{-\pi i /3}
  F_5.
\end{equation}

Let $V$ be the
4-dimensional $p$-adic Scholl space which contains $\langle F_1, F_5
\rangle$ as a subspace and their dual $\langle F_1^{\vee},
F_5^{\vee}\rangle$ as a quotient (cf. \cite{sch85b}). Let
$J_{-2}=\zeta W_2, J_{-3}=\frac 1{\sqrt 3}(2\zeta-1),
J_{6}=J_{-2}J_{-3}$. With respect to a suitable basis $J_{-2}$ and $
J_{-3}$ can be represented by the following matrices respectively:
$$\begin{pmatrix}0&\frac12 iI_2\\ 2 iI_2&0
\end{pmatrix} \quad {\rm and} \quad  \begin{pmatrix}iI_2&0\\0&-iI_2
\end{pmatrix}.$$ Thus $J_{-2}$ and $J_{-3}$ generate a quaternion algebra
over $\Q$ with $J_{-2}^2=-I, J_{-3}^2=-I,
J_{-3}J_{-2}=-J_{-2}J_{-3}$.  It is easy to check that
\begin{itemize}
\item $F_1$ and $F_5$ are eigenvectors of $J_{-3}$ with eigenvalues $\pm i$;
\item $F_1\pm 2  F_5$ are eigenvectors of $J_{-2}$ with eigenvalues
$\pm i $;
\item $F_1\pm 2i F_5$ are eigenvectors  of
$J_{6}$ with eigenvalues $\pm i$.
\end{itemize}
\bigskip

We now turn to the $\ell$-adic side.
 Under the action of $W_2$, $B \mapsto
-8/B$. This gives rise to an involution on the modular curves for $\G^1(6) $ since $W_2$ normalizes $\G^1(6)$. Consider two elliptic surfaces $\mathcal E_{\G^1(6) }$ and $\mathcal E_{W_2^{-1}\G^1(6) W_2}$ fibred over the modular curve $X_{\G^1(6)}$ parametrized by $B$ defined by the following equations:
\begin{eqnarray*}
  \mathcal E_{\G^1(6)} &&: \quad Y^2=X^3+a(B)X^2+b(B)X, \quad\\ \mathcal E_{W_2^{-1}\G^1(6)
W_2}&&: \quad Y^2=X^3-2a(\frac{-8}B)X^2+ \left (a(\frac{-8}B)^2-4b(\frac{-8}B)\right )X,
\end{eqnarray*}   where
$$a(B):=\frac{2}{27}-\frac{5}{27}B-\frac{1}{108}B^2 \quad \text{and} \quad b(B):=\frac{1}{729}+\frac{1}{243}B+\frac{1}{729}B^3+\frac{1}{243}B^2.$$
Then $$(X,Y, B) \mapsto \left (\frac {Y^2}{X^2},
\frac{Y(b(B)-X^2)}{X^2}, \frac{-8}B \right )$$ gives a $2$-isogeny from $\mathcal
E_{\G^1(6) }$ to $\mathcal E_{W_2^{-1}\G^1(6)  W_2}$.
\medskip

By letting $t^6=B$ we obtain an explicit defining equation for the
elliptic surface $\mathcal E_\G$ fibred over the modular curve of
$\G$. A similar argument yields the following degree-2 isogeny
between  $\mathcal E_{\G}$ and $\mathcal E_{W_2^{-1}\G W_2}$:  Again,
$W_2$ normalizes $\G$ and it gives rise to the following  map
{ \begin{equation}W_2: \quad (X,Y, t) \mapsto \left ( \frac{Y^2}{X^2},
\frac{Y(b(t^6)-X^2)}{X^2}, \sqrt[6]{-8}/t\right ).
\end{equation}}
Also on  $\mathcal E_{\G}$ and $\mathcal E_{W_2^{-1}\G W_2}$, we
have
$$\zeta: \quad (X,Y, t) \mapsto (X,Y, e^{- \pi i/3}t).$$
It is straightforward to check that
\begin{equation}\label{eq:5.4}\zeta W_2\zeta=W_2,
\end{equation} where the $\zeta$ on the left acts on $\mathcal
E_{W_2^{-1}\G W_2}$ and the one on the right acts on $\mathcal
E_{\G}$.

The composition $T_{-2}=\zeta\circ W_2$ is an isogeny defined over $\Q_{\ell}(\sqrt{-2})$. Let
$\widehat T_{-2}$ denote its dual isogeny.

We first note that the above map $T_{-2}$ induces a natural
homomorphism $g_1$ from the parabolic cohomology $H^1(X_{\G}\otimes
\overline{\BQ}, i_*R^1h_*^0\BQ_\ell)$ to that of $X_{W_2^{-1}\G W_2}$. Here
$i: \frak H/\G \rightarrow X_{\G}=(\frak H/\G)^*$ is the inclusion
map and $h: \mathcal E_\G \rightarrow (\frak H/\G)^*$ is the natural
projection and $h^0$ is the restriction of $h$ to $\frak H/\G$.  We
further let $w_2: X_{\G} \rightarrow X_{W_2^{-1}\G W_2}$ denote the
isomorphism which sends $t$ to
$\sqrt[6]{-8}/t$. Compared with $T_{-2}$, this map is on the
base curve level.
Using $w_2$, we have a natural isomorphism
$$g_2: H^1(X_{\G}\otimes \overline{\BQ}, i_*R^1h_*^0\BQ_\ell)\cong
H^1(X_{W_2^{-1}\G W_2}\otimes \overline{\BQ},
(w_2)_*i_*R^1h_*^0\BQ_\ell),$$ and the cohomology on the right is isomorphic to the parabolic cohomology attached to weight 3 cusp
forms of ${W_2^{-1}\G W_2}$. The composition $g_2^{-1}\circ g_1$
is an endomorphism on $H^1(X_{\G}\otimes \overline{\BQ},
i_*R^1h_*^0\BQ_\ell)$ and is denoted by $T_{-2}$ again. On $H^1(X_{\G}\otimes
\overline{\BQ}, i_*R^1h_*^0\BQ_\ell)$ the map $(T_{-2})^2=-2~id$ is induced by the fiber-wise  multiplication by $2$  map
$\widehat T_{-2} \circ T_{-2}$.

Denote by $\rho_\ell^{\text{new}}$ the 4-dimensional $\ell$-adic Scholl representation
of $G$ corresponding to the  ``new" forms of $S_3(\G)$ (cf.
\cite{long061}).   Let $J_{-2} = \frac{1}{\sqrt 2} T_{-2}$ and let  $J_{-3}$ denote
$\frac 1{\sqrt 3}(2\zeta-1)^*$  on $\rho_\ell^{\text{new}}$. As linear operators on the representation space of
$\rho_\ell^{\text{new}}$, $J_{-2}$ and $J_{-3}$ generate a quaternion group.  We know from the
above discussions that the actions of $J_{\d}$ are defined
over $\Q(\sqrt{\d})$ for $\d=-2,-3$.  Choose the field $F $ in Theorem \ref{QMdescent} to be $\Q(\sqrt 2, \sqrt 3)$.
 The space of $\rho_{\ell}^{\text{new}}$   over fields in $\Q_\ell \otimes F$ is endowed with QM over $\Q(\sqrt {-2}, \sqrt{-3})$. By Theorem \ref{QMdescent},  we conclude the automorphy of $\rho_\ell^{\text{new}}$.

We have also seen that $J_{-2}$ and $J_{-3}$ act on the space $\langle F_1, F_5 \rangle$ and the space $V_p$
admits QM over $\Q(\sqrt {-2}, \sqrt {-3})$. Since $J_{-2}$ and $J_{-3}$  arise from normalizers of $\G^1(6)$, by Theorems \ref{QMdescent} and \ref{ASD}, we have
\begin{theorem}
The representation $\rho_\ell^{\text{new}}$ is automorphic, and the ASD congruences hold
on the space $\langle F_1, F_5 \rangle$.
\end{theorem}

In particular, $\rho_\ell^{\text{new}}$ is isomorphic to  $\eta_{-2}\otimes \text{Ind}_{G_{\Q(\sqrt{-2})}}^{G} \chi_{-2}^{-1}$, { which corresponds to an automorphic representation of} $\t{GL}_2(\BA_\Q) \times \t{GL}_2(\BA_\Q)$, where $\chi_{-2}$ is a character of $G_{\Q(\sqrt{-2})}$. Similar to \cite{all05}, we conclude this section by exhibiting, from numerical data, a congruence modular form $f$ which gives rise to $\eta$ with $\chi_{-2}=\xi_1 \xi_{2}$ being the product of the following two characters of $G_{\Q(\sqrt{-2})}$: $\xi_1$ is a quartic character of
$\Gal(\Q(\sqrt{-2},\sqrt{1+\sqrt 2})/\Q(\sqrt{-2}))$ and
$\xi_2$ is the quadratic character of $\Gal(\Q(\sqrt{1+\sqrt {-2}})/\Q(\sqrt{-2}))$.

To motivate the expression for $f$, we list explicit factorizations of the characteristic polynomials of $\text{Frob}_p$ at small primes.

$$\begin{tabular}{|c|c|c|c}\hline
$p$&Char. poly.& Factorization\\ \hline
5&$x^4+4x^2+5^4$&$(x^2-3x\sqrt{-6}-25)(x^2+3x\sqrt{-6}-25)$\\ \hline
7&$x^4+10x^2+7^4$&$(x^2-6x\sqrt{-3}-49)(x^2+6x\sqrt{-3}-49)$\\
\hline 11&$x^4-170x^2+11^4$& $
(x^2-6x\sqrt{-2}-121)(x^2+6x\sqrt{-2}-121)$ \\ \hline
13&$x^4-230x^2+13^4$& $(x^2-6x\sqrt{-3}-13^2)(x^2+6x\sqrt{-3}-13^2)$\\
\hline 17&$x^4-128 x^2+17^4$& $(x^2-15 x \sqrt{-2}-17^2) (x^2+15 x
\sqrt{-2}-17^2)$\\ \hline 19& $(x^2+20 x+19^2)^2$ & $(x^2+20
x+19^2)^2$\\ \hline 23&$x^4-842 x^2+23^4$& $(x^2+6 x\sqrt{-6}-23^2)
(x^2-6 x \sqrt{-6}-23^2)$\\ \hline 29&$ x^4-332x^2+29^4$&
$(x^2+15x\sqrt{-6}-841)(x^2-15x\sqrt{-6}-841)$ \\ \hline
\end{tabular}$$

The congruence form $f(z)$ is a Hecke eigenform in $S_3(\G_0(576),(-6/\cdot))$ whose Fourier coefficients occur at positive integers coprime to $24$. Express it as a linear combination of forms $f_j$ whose Fourier coefficients occur at integers congruent to $j$ mod $24$ as follows:
$$ f(z/24):=f_1+3jk\cdot f_5+6k\cdot f_7+6j\cdot f_{11}+6ik\cdot f_{13}+3ij\cdot f_{17}-4i \cdot f_{19}-6ijk\cdot f_{23},$$ where $i=\sqrt{-1}, j=\sqrt{2}$ and $k=\sqrt{3}$. Here
\begin{eqnarray*}
  f_5=\frac{\eta(z)^5\eta(6z)^5}{\eta(3z)^2\eta(12z)^2},
   &&f_7=\frac{\eta(z)^5\eta(4z)\eta(6z)^2}{\eta(2z)\eta(12z)},\\
   f_{13}=\frac{\eta(z)^4\eta(2z)^2\eta(3z)\eta(12z)}{\eta(4z)\eta(6z)},&&
    f_{23}=\frac{\eta(z)^5\eta(12z)^2}{\eta(6z)},
     \end{eqnarray*} which span a space invariant under $W_9$ and $W_{64}$. The remaining forms
$f_1, f_{11}, f_{17}, f_{19}$ are uniquely determined by applying the Hecke operators to the eigenform $f$. Listed below are their initial Fourier coefficients:
\begin{eqnarray*}
  f_1&=&q^{1/24}(q+ 29q^2 +59q^3+ 20q^4 +40q^5 -49q^6 -270q^7+ 61q^8+\cdots)\\
  f_{11}&=&q^{11/24}(q+ 9q^2+ 4q^3+ 15q^4 -20q^5+ 6q^6 -45q^7+ 6q^8+\cdots)\\
  f_{17}&=&q^{17/24}(5q+ 17q^2+ 18q^3+ 3q^4 -15q^5 -25q^6 -36q^7 -72q^8+\cdots)\\
  f_{19}&=&q^{19/24}(5q+ 10q^2+ 25q^3 -27q^4+ 27q^5 -43q^6 -40q^7 -45q^8+\cdots).
\end{eqnarray*} Choosing different signs of $i,j,k$ gives rise to 8 eigenforms that are conjugate to each other under $G$.

\end{document}